\documentclass[a4paper,10pt]{amsart}

\usepackage{amsmath}
\usepackage{amsthm}
\usepackage{amssymb}
\usepackage{bbm}

\usepackage[unicode]{hyperref}
\usepackage{textgreek}

\usepackage{mathrsfs}
\usepackage[all]{xy}

\newcommand{\set}[2]{\left\{ #1 \mid #2 \right\}}

\newcommand{\PP}{\mathbb{P}}

\newcommand{\Isom}{\operatorname{Isom}}

\newcommand{\aideal}{\mathfrak{a}}

\newcommand{\Gr}{G}

\newcommand{\MC}{\mathsf{MC}}
\newcommand{\Tw}{\mathsf{Tw}}

\newcommand{\interior}{\operatorname{int}}

\newcommand{\map}{map}

\newcommand{\GL}{GL}

\newcommand{\ad}{\operatorname{ad}}
\newcommand{\dquot}{/\hspace{-3pt}/}

\newcommand{\HL}{\mathcal{L}}

\newcommand{\QQ}{\mathbb{Q}}
\newcommand{\RR}{\mathbb{R}}
\newcommand{\ZZ}{\mathbb{Z}}

\newcommand{\monoid}{\mathcal{H}}
\newcommand{\Gmonoid}{\mathcal{G}}

\newcommand{\univ}{\gamma}

\newcommand{\FF}{\mathcal{F}}

\newcommand{\pdideal}{\mathfrak{d}}

\newcommand{\tensor}{\otimes}

\newcommand{\ctensor}{\widehat{\tensor}}
\newcommand{\Hom}{\operatorname{Hom}}

\newcommand{\Diff}{\operatorname{Diff}}
\newcommand{\tDiff}{\widetilde{\Diff}}
\newcommand{\Emb}{\operatorname{Emb}}

\newcommand{\aut}{aut}

\newcommand{\Der}{\operatorname{Der}}

\newcommand{\gl}{\mathfrak{g}}
\newcommand{\hl}{\mathfrak{h}}
\newcommand{\al}{\mathfrak{a}}

\newcommand{\homl}{\mathfrak{hom}}

\newcommand{\CP}{\mathbb{C}\mathrm{P}}

\newcommand{\SDR}[5]{\xymatrix{*[r]{#1} \ar@<1ex>[r]^-{#3} \ar@(ul,dl)[]_{#5} & #2 \ar@<1ex>[l]^-{#4}}}
\newcommand{\bigSDR}[5]{\xymatrix{*[r]{#1} \ar@<1ex>[rr]^-{#3} \ar@(ul,dl)[]_{#5} && #2 \ar@<1ex>[ll]^-{#4}}}
\newcommand{\bigbigSDR}[5]{\xymatrix{*[r]{#1} \ar@<1ex>[rrr]^-{#3} \ar@(ul,dl)[]_{#5} &&& #2 \ar@<1ex>[lll]^-{#4}}}

\newcommand{\nocontentsline}[3]{}
\newcommand{\tocless}[2]{\bgroup\let\addcontentsline=\nocontentsline#1{#2}\egroup}

\newtheorem{theorem}{Theorem}[section]

\newtheorem{proposition}[theorem]{Proposition}
\newtheorem{corollary}[theorem]{Corollary}
\newtheorem{lemma}[theorem]{Lemma}

\theoremstyle{definition}
\newtheorem{definition}[theorem]{Definition}
\newtheorem{remark}[theorem]{Remark}

\title{Characteristic classes for families of bundles}
\author{Alexander Berglund}
\address{Department of Mathematics\\
Stockholm University\\
SE-106 91 Stockholm\\
Sweden}
\email{alexb@math.su.se}

\begin{document}
\begin{abstract}
The generalized Miller-Morita-Mumford classes of a manifold bundle with fiber $M$ depend only on the underlying \emph{$\tau_M$-fibration}, meaning the family of vector bundles formed by the tangent bundles of the fibers. This motivates a closer study of the classifying space for $\tau_M$-fibrations, $B\aut(\tau_M)$, and its cohomology ring, i.e., the ring of characteristic classes of $\tau_M$-fibrations.

For a bundle $\xi$ over a simply connected Poincar\'e duality space, we construct a relative Sullivan model for the universal orientable $\xi$-fibration together with explicit cocycle representatives for the characteristic classes of the canonical bundle over its total space. This yields tools for computing the rational cohomology ring of $B\aut(\xi)$ as well as the subring generated by the generalized Miller-Morita-Mumford classes.

To illustrate, we carry out sample computations for spheres and complex projective spaces. We discuss applications to tautological rings of simply connected manifolds and to the problem of deciding whether a given $\tau_M$-fibration comes from a manifold bundle.
\end{abstract}

\maketitle

\tableofcontents

\section{Introduction}
The generalized Miller-Morita-Mumford classes, or tautological classes, are characteristic classes of manifold bundles that play an important role in the study of the cohomology of moduli spaces of manifolds \cite{MW,GRW}. The tautological class $\kappa_c$, associated to a characteristic class $c$ of oriented vector bundles, is defined by its evaluation on an oriented manifold bundle,
$$M\to E\xrightarrow{\pi} B,$$
namely
$$\kappa_c(\pi) = \int_M c(T_\pi E) \in H^*(B),$$
i.e., the class $\kappa_c(\pi)$ is obtained by integration along the fiber of the characteristic class $c$ evaluated on the fiberwise tangent bundle $T_\pi E$.

This paper studies characteristic classes that, like the tautological classes, are defined using only tangential data, viz.~the fiberwise tangent bundle and its characteristic classes, and homotopy theory.

For a bundle $\xi$ over a space $X$, we define a \emph{$\xi$-fibration} to be a pair $(\pi,\zeta)$, where
\begin{itemize}
\item $\pi\colon E\to B$ is a fibration whose fibers are homotopy equivalent to $X$, and
\item $\zeta$ is a bundle over $E$ (the `total bundle') whose restriction to each fiber is equivalent to $\xi$ in an appropriate sense.
\end{itemize}

In other words, a $\xi$-fibration is a \emph{family of bundles} $\{\zeta_b\}$ parameterized by the space $B$ such that $\zeta_b \sim\xi$ for every $b\in B$.

Every smooth manifold bundle $\pi\colon E\to B$ with fiber $M$ gives rise to a $\tau_M$-fibration $(\pi,\zeta)$ where $\zeta_b$ is the tangent bundle of the fiber $E_b = \pi^{-1}(b)$, i.e., $\zeta$ is the fiberwise tangent bundle $T_\pi E$.

The base of the universal $\xi$-fibration may be identified with $B\aut(\xi)$, the classifying space of the topological monoid $\aut(\xi)$ of automorphisms of $\xi$ that cover a self-homotopy equivalence of $X$. Therefore, the cohomology ring of $B\aut(\xi)$ may be thought of as the ring of characteristic classes of $\xi$-fibrations.

The main result of the paper, Theorem \ref{thm:sullivan model}, is the construction of a relative Sullivan model for the universal orientable $\xi$-fibration together with explicit formulas for cocycle representatives of the characteristic classes of its total bundle. This can be used to compute $H^*(B\aut(\xi);\QQ)$, and the subring $R^*(\xi)$ generated by the tautological classes, for bundles $\xi$ over simply connected Poincar\'e duality spaces.

We will spend the rest of this introduction discussing sample calculations done using Theorem 3.8 and some applications. Full details are given in \S4.

\begin{theorem} \label{thm:even spheres}
Consider the oriented tangent bundle $\tau_{S^m}$ of an even dimensional sphere $S^m$. The ring of characteristic classes of $\tau_{S^{m}}$-fibrations may be identified with the polynomial ring
$$H^*(B\aut(\tau_{S^{m}});\QQ) = \QQ[\kappa_{ep_1},\ldots,\kappa_{ep_k}, \kappa_{p_r},\ldots,\kappa_{p_k}],$$
where $m = 2k$ and $r$ is the smallest integer such that $4r>m$. 
In particular, all characteristic classes of $\tau_{S^{m}}$-fibrations are tautological.
\end{theorem}

Our calculation for odd dimensional spheres, to be presented next, informs the following definition.
For a class $c\in H^*(BSO(m))$ of degree $<m$, we define a characteristic class $\alpha_c$ of $\tau_{S^m}$-fibrations $(\pi\colon E\to B,\zeta)$ by
$$\pi^*(\alpha_c(\pi,\zeta)) = c(\zeta) \in H^*(E).$$
This uniquely defines $\alpha_c(\pi,\zeta)\in H^*(B)$, because $\pi^*\colon H^*(B) \to H^*(E)$ is an isomorphism in degrees $<m$.

Next, recall that every spherical fibration $S^{m}\to E\xrightarrow{\pi} B$ has an associated Euler class $e(\pi)\in H^{m+1}(B)$.
The Euler class is a characteristic class of spherical fibrations and in particular of $\tau_{S^{m}}$-fibrations.

Let $A^*(\tau_{S^{m}})$ denote the subring of $H^*(B\aut(\tau_{S^{m}});\QQ)$ generated by the $\alpha$-classes and the Euler class.

\begin{theorem} \label{thm:odd spheres}
Consider the oriented tangent bundle $\tau_{S^{m}}$ of an odd dimensional sphere $S^m$. Let $m=2k+1\geq 3$ and let $r$ be the smallest integer such that $4r>m$.
\begin{enumerate}
\item The cohomology ring $H^*(B\aut(\tau_{S^{m}});\QQ)$ is additively the direct sum
$$A^*(\tau_{S^{m}}) \oplus R^*(\tau_{S^{m}}),$$
and the multiplication is determined by
\begin{align*}
\alpha_a \kappa_b & = \kappa_{ab}, \\
e \kappa_b & = 0,
\end{align*}
for all $a,b\in H^*(BSO(m))$ with $|a|<m$.
In particular, the ring of characteristic classes of $\tau_{S^{m}}$-fibrations is generated by the $\alpha$-classes, the Euler class, and the $\kappa$-classes.

\item The ring generated by the $\alpha$-classes and the Euler class may be identified with the polynomial ring
$$A^*(\tau_{S^{m}}) = \QQ[\alpha_{p_1},\ldots,\alpha_{p_{r-1}},e].$$

\item The tautological ring is isomorphic to the ring of exact K\"ahler differential forms on $\QQ[p_1,\ldots,p_k]$, with respect to a formal differential $d$ of degree $-m$ which is linear over $\QQ[p_1,\ldots,p_{r-1}]$,
$$R^*(\tau_{S^{m}}) \cong d \Omega_{\QQ[p_1,\ldots, p_k]|\QQ[p_1,\ldots,p_{r-1}]}^*,$$
through an isomorphism that sends $\kappa_c$ to $dc$ for all $c\in \QQ[p_1,\ldots,p_k]$.
\end{enumerate}
\end{theorem}

For concreteness, let us look closer at what this means for $S^3$.

\begin{corollary}
The tautological ring $R^*(\tau_{S^3})$ of the tangent bundle of $S^3$ is spanned by the $\kappa$-classes associated to the Hirzebruch $L$-classes,
$$\kappa_{\HL_1}, \,\,\kappa_{\HL_2},\,\, \cdots.$$
The multiplication is trivial.
\end{corollary}

\begin{proof}
The ring of exact K\"ahler forms over a polynomial ring has a non-trivial algebraic structure in general, but for $m =3$ the ring
$$R^*(\tau_{S^3}) = d \Omega_{\QQ[p_1]|\QQ}^*.$$
has basis $p_1^{i-1}dp_1$ for $i=1,2,\ldots$ and the multiplication is trivial.
The Hirzebruch $L$-classes in the cohomology of $BSO(3)$ are given by $\HL_i = b_i p_1^i$,
for certain non-zero rational numbers $b_i$.
Through the isomorphism in the theorem, $\kappa_{\HL_i}$ corresponds to
$d\HL_i = i b_i p_1^{i-1}dp_i$,
showing the classes $\kappa_{\HL_i}$ are non-zero and span $R^*(\tau_{S^3})$.
\end{proof}
The classes $\kappa_{\HL_i}$ may be thought of as obstructions for extending a given $\tau_{S^m}$-fibration to a fiber bundle;
it is a consequence of the family signature theorem that $\kappa_{\HL_i}(\pi,\zeta) = 0$ for all $i>m/4$ if $(\pi,\zeta)$ is the $\tau_{S^m}$-fibration associated to a smooth manifold bundle with fiber $S^m$ (see \cite[Theorem A.2]{KRW}).

Define $B\aut(\tau_{S^m})_L$ to be the homotopy fiber of the map
$$L\colon B\aut(\tau_{S^m}) \to \prod_{m/4<i\leq  m/2} K(\QQ,4i-m)$$
that records the classes $\kappa_{\HL_i}$ in the indicated range. It may be thought of as the classifying space of $\tau_{S^m}$-fibrations with trivializations of the classes $\kappa_{\HL_i}$.

Every oriented vector bundle $\pi\colon E\to B$ of dimension $m+1$ has an associated sphere bundle
$$S^m \to S(E) \to B.$$
This is an $SO(m+1)$-bundle, so it can be equipped with a fiberwise tangent bundle $\zeta\to S(E)$ making it into a $\tau_{S^m}$-fibration, and the map that classifies it factors through $B\aut(\tau_{S^m})_L$.

While the calculations for $m$ even and $m$ odd are quite different, they both lead to the following result (for $m$ odd we do not suggest that it is an obvious consequence of the statement of Theorem \ref{thm:odd spheres} but it follows from the methods that prove it).
\begin{theorem} \label{thm:spheres}
The map
\begin{equation} \label{eq:so}
BSO(m+1) \to B\aut(\tau_{S^m})_L,
\end{equation}
induced by taking the oriented sphere bundle of the universal oriented vector bundle, is a rational homotopy equivalence.
\end{theorem}

As an application, this yields criteria for when a $\tau_{S^m}$-fibration is rationally equivalent to an $SO(m+1)$-bundle with fiber $S^m$. 

\begin{corollary} \label{cor:sphere bundle}
A $\tau_{S^m}$-fibration
$$S^m \to E \xrightarrow{\pi} B,\quad \zeta \to E,$$
is rationally equivalent to an $SO(m+1)$-bundle if and only if
$$\kappa_{\HL_i}(\pi,\zeta) = 0$$
for $m/4<i\leq m/2$.
\end{corollary}

To be precise, what we mean when we say that a $\tau_{S^m}$-fibration $(\pi,\zeta)$ is rationally equivalent to an $SO(m+1)$-bundle is that there exists a rational homotopy equivalence $f\colon B' \to B$ such that the pulled back $\tau_{S^m}$-fibration $f^*(\pi,\zeta)$ is equivalent to the underlying $\tau_{S^m}$-fibration of an $SO(m+1)$-bundle over $B'$ with fiber $S^m$.

As another application, we note that computations over $B\aut(\tau_M)$ can yield information about the tautological ring of $M$ in the sense of \cite{GGRW}, i.e., the subring $R^*(M)$ of $H^*(B\Diff^+(M);\QQ)$ generated by the $\kappa$-classes.
Indeed, the map
$$B\Diff^+(M) \to B\aut(\tau_M)$$
that classifies the $\tau_M$-fibration associated to the universal smooth oriented $M$-bundle induces a surjective ring homomorphism
$$R^*(\tau_M) \to R^*(M).$$
Thus, in principle, the ring $R^*(M)$ can be carved out of $R^*(\tau_M)$ by imposing further relations.
As a simple illustration of this point, we give an alternative calculation of the tautological ring of $S^m$ for $m$ even, cf.~\cite[Theorem 1.1(i)]{GGRW}.

\begin{corollary} \label{cor:tautological ring}
The ring homomorphism
$$H^*(B\aut(\tau_{S^{m}})_L;\QQ) \to H^*(B\Diff^+(S^{m});\QQ).$$
is split injective. For $m=2k$ even, the image is the tautological ring $R^*(S^{m})$, and
$$R^*(S^{m}) \cong R^*(\tau_{S^m})/(\kappa_{\HL_r},\ldots,\kappa_{\HL_k}) \cong \QQ[\kappa_{e\HL_1},\ldots,\kappa_{e\HL_k}].$$
\end{corollary}

\begin{proof}
The map \eqref{eq:so} factors as
$$BSO(m+1) \to B\Diff^+(S^m) \to B\aut(\tau_{S^m})_L,$$
inducing
$$H^*(B\aut(\tau_{S^{m}})_L;\QQ) \to H^*(B\Diff^+(S^m);\QQ) \to H^*(BSO(m+1);\QQ).$$
The composite is a ring isomorphism by Theorem \ref{thm:spheres}. This proves the first claim.

Now assume $m$ is even. Since all characteristic classes of $\tau_{S^m}$-fibrations are tautological, the image is $R^*(S^m)$.
As will be clear from the proof, one can replace the Pontryagin classes by the Hirzebruch $L$-classes in Theorem 1.1, so that
$$H^*(B\aut(\tau_{S^{m}});\QQ) = R^*(\tau_{S^{m}}) = \QQ[\kappa_{e\HL_1},\ldots,\kappa_{e\HL_k}, \kappa_{\HL_r},\ldots,\kappa_{\HL_k}].$$
Since $\kappa_{\HL_r},\ldots,\kappa_{\HL_k}$ is a regular sequence, it follows that
$$H^*(B\aut(\tau_{S^{m}})_L;\QQ) \cong R^*(\tau_{S^{m}})/(\kappa_{\HL_r},\ldots,\kappa_{\HL_k}) \cong \QQ[\kappa_{e\HL_1},\ldots,\kappa_{e\HL_k}].$$
\end{proof}

For another sample calculation, we turn to complex projective spaces. Consider a fibration
$$\CP^n \to E\xrightarrow{\pi} B$$
which is orientable in the sense that $\pi_1(B)$ acts trivially on the cohomology of the fiber.
For every choice of generator $\omega\in H^2(\CP^n;\QQ)$, there is a unique cohomology class $\omega_{fw}(\pi) \in H^2(E;\QQ)$ (the `coupling class') such that 
$\omega_{fw}(\pi)|_{\CP^n} = \omega$ and
$$\int_{\CP^n} \omega_{fw}(\pi)^{n+1} = 0.$$
For definiteness, we fix the generator $\omega = -c_1(\gamma^1)$, the negative of the first Chern class of the canonical line bundle $\gamma^1$.

By a standard argument, it follows that $H^*(E;\QQ)$ is a free $H^*(B;\QQ)$-module with basis $1,\, \omega_{fw}(\pi), \ldots,\, \omega_{fw}(\pi)^n$.
A key observation is that $\omega_{fw}(\pi)$ is natural in the fibration $\pi$ (see Lemma \ref{lemma:fw generator}).
Inspired by Grothendieck's approach to Chern classes (cf.~\cite{Grothendieck} or \cite[\S20]{BT}), we can then define characteristic classes of orientable 
$\CP^n$-fibrations $\pi\colon E\to B$,
$$a_i(\pi) \in H^{2i}(B;\QQ),$$
by postulating the equality
$$\omega_{fw}(\pi)^{n+1} + a_2(\pi) \cdot \omega_{fw}(\pi)^{n-1} +\ldots + a_{n+1}(\pi)\cdot 1 = 0$$
in the cohomology of the total space (we set $a_0(\pi) =1$ and $a_1(\pi) = 0$)\footnote{For the precise relation between the classes $a_i$ and the Chern classes in the case when the fibration arises through projectivization of a complex vector bundle, see Proposition \ref{prop:a vs chern}.}.

Similarly, if $\xi$ is a bundle over $\CP^n$ with structure group $G$, then for every class $x\in H^\ell(BG;\QQ)$ we can define characteristic classes
$$x_{|j}(\pi,\zeta) \in H^{\ell-2j}(B;\QQ)$$
of orientable $\xi$-fibrations
$$\CP^n \to E \xrightarrow{\pi} B,\quad \zeta \to E,$$
by postulating the equality 
$$x(\zeta) =  x_{|n}(\pi,\zeta) \cdot \omega_{fw}(\pi)^n + \ldots + x_{|1}(\pi,\zeta) \cdot \omega_{fw}(\pi) + x_{|0}(\pi,\zeta)\cdot 1 = 0$$
in the cohomology of the total space.

\begin{theorem} \label{thm:pu}
Let $\xi$ be a bundle over $\CP^n$ with structure group $G$. Assume that $G$ is connected and that the cohomology of $BG$ is an evenly graded polynomial ring, say $H^*(BG;\QQ) = \QQ[p_1,p_2,\ldots]$ with $|p_i|=2r_i$.
\begin{enumerate}
\item The ring of characteristic classes of orientable $\xi$-fibrations,
$$H^*(B\aut_\circ(\xi);\QQ),$$
may be identified with the polynomial ring
$$\QQ[a_2,\ldots,a_{n+1},p_{i|j}],$$
generated by $a_2,\ldots, a_{n+1}$ and $p_{i|j}$ for all $i$ and all $j$ such that $0\leq j < r_i$.

\item The ring of characteristic classes of $\xi$-fibrations may be identified with the invariant subring
$$H^*(B\aut(\xi);\QQ) = \QQ[a_2,\ldots,a_{n+1},p_{i|j}]^{\Gamma(\xi)},$$
where
$$\Gamma(\xi) \cong \left\{ \begin{array}{cc} \ZZ/2\ZZ, & c^*(\xi)\cong \xi , \\ 0, & c^*(\xi) \not\cong \xi, \end{array} \right.$$
where $c\colon \CP^n \to \CP^n$ denotes complex conjugation, and where the non-trivial element of $\Gamma(\xi)$
acts by
$$a_k\mapsto (-1)^k a_k, \quad p_{i|j}\mapsto (-1)^j p_{i|j},$$
in the case when $c^*(\xi) \cong \xi$.
\end{enumerate}
\end{theorem}

For example, since complex conjugation is an orientation preserving diffeomorphism of $\CP^2$ we have $c^*(\tau_{\CP^2}^\RR) \cong \tau_{\CP^2}^{\RR}$, where the latter denotes the underlying oriented vector bundle of the complex tangent bundle, whence
$$
H^*(B\aut(\tau_{\CP^2}^\RR);\QQ) \cong \QQ[a_2, p_{1|0}, e_{|0}, p_{1|1}^2, p_{1|1}e_{|1} ,e_{|1}^2, p_{1|1}a_3, e_{|1}a_3, a_3^2].
$$
This ring is abstractly isomorphic to
$$\QQ[u,v,w,a,b,c,d,e,f]/(ac-b^2,af-d^2,cf-e^2).$$
In particular, it is a complete intersection of Krull dimension $6$ and embedding dimension $9$. 

By counting dimensions, one quickly realizes that there are not enough tautological classes to generate $H^*(B\aut_\circ(\xi);\QQ)$ in general. However, if we extend the set of tautological classes by defining
$$\kappa_{\omega^n c}(\pi,\zeta) = \int_{\CP^n} \omega_{fw}(\pi)^n c(\zeta),$$
for $n\geq 0$ and $c\in H^*(BG;\QQ)$, then we have the following.

\begin{theorem} \label{thm:extended tautological}
With hypotheses as in Theorem \ref{thm:pu}, the ring of characteristic classes of orientable $\xi$-fibrations is a polynomial ring in the classes
$$\kappa_{\omega^{n+2}},\ldots, \kappa_{\omega^{2n+1}},$$
$$\kappa_{\omega^\ell p_i},\quad n-r_i+1\leq \ell \leq n,\quad i=1,2,\ldots.$$
In particular, all such characteristic classes are tautological in the extended sense.
\end{theorem}

\begin{remark}
For the $\tau_{\CP^n}$-fibration $(\pi,\zeta)$ associated to a symplectic $\CP^n$-bundle, the class $\kappa_{\omega^{n+k}c^I}(\pi,\zeta)$, where $c^I = c_1^{m_1}\ldots c_n^{m_n}$ for a multi-index $I=(m_1,\ldots,m_n)$, agrees with the class $\mu_{k,I}$ defined by K\c edra-McDuff \cite[p.147]{KM}.
\end{remark}

\begin{remark}
By letting $\xi$ be the trivial bundle with fiber a point, we recover a result of Kuribayashi \cite{Kuribayashi-elliptic} on the ring of characteristic classes of orientable $\CP^n$-fibrations as a special case of Theorem \ref{thm:extended tautological}: $H^*(B\aut_\circ(\CP^n);\QQ) = \QQ[\mu_2,\ldots,\mu_{n+1}]$.
\end{remark}

As an application of Theorem \ref{thm:pu}, we establish necessary and sufficient conditions for when a $\tau_{\CP^n}$-fibration is rationally equivalent to a $PU(n+1)$-bundle. To do this, we first equip every orientable $\CP^n$-fibration with a cohomological stand-in for a fiberwise tangent bundle.

\begin{definition}
For an orientable fibration
$$\CP^n \to E \xrightarrow{\pi} B,$$
we define the \emph{fiberwise Chern classes} $c_1^{fw}(\pi),\ldots, c_n^{fw}(\pi)$ by the formula
\begin{equation} \label{eq:fiberwise Chern classes}
c_i^{fw}(\pi) = \sum_{j=0}^i \binom{n+1-j}{i-j}a_j(\pi) \omega_{fw}(\pi)^{i-j}\in H^{2i}(E;\QQ).
\end{equation}
\end{definition}
We refer to the computation in Proposition \ref{prop:ucd} for the origin of this formula.
Let $B\aut(\tau_{\CP^n})^c$ denote the classifying space for $\tau_{\CP^n}$-fibrations $(\pi,\zeta)$ with a trivialization of the `Chern difference'
$$cd_i(\pi,\zeta) = c_i^{fw}(\pi) - c_i(\zeta)$$
for every $i$.

\begin{theorem} \label{thm:cd}
The universal $PU(n+1)$-bundle with fiber $\CP^n$ has trivial Chern differences and the induced map
$$BPU(n+1) \to B\aut(\tau_{\CP^n})^c,$$
is a rational homotopy equivalence.
\end{theorem}

\begin{corollary} \label{cor:pun}
A $\tau_{\CP^n}$-fibration is rationally equivalent to a $PU(n+1)$-bundle if and only if its Chern differences are trivial.
\end{corollary}

\begin{remark}
Since $c_1^{fw}(\pi)  = (n+1)\omega_{fw}(\pi)$, it follows immediately from Theorem \ref{thm:extended tautological} that all characteristic classes of $\tau_{\CP^n}$-fibrations with trivial first Chern difference are tautological.
\end{remark}

\begin{remark}
The map $B\aut(\tau_{\CP^n})^c \to B\aut_\circ(\CP^n)$ is a rational equivalence (see Remark \ref{rmk:bpun}). Using this and the preceding remarks, we see that the ring of characteristic classes of $\tau_{\CP^n}$-fibrations with trivialized Chern differences (and the ring of characteristic classes of $PU(n+1)$-bundles with fiber $\CP^n$) may be identified with the polynomial ring in the classes $\kappa_{c_1^{n+2}},\ldots,\kappa_{c_1^{2n+1}}$.
\end{remark}

There are similar results for $\tau_{\CP^n}^\RR$-fibrations $(\pi,\zeta)$, where the total bundle $\zeta$ is an oriented vector bundle instead of a complex vector bundle. We define fiberwise Pontryagin classes and a fiberwise Euler class by
\begin{align*}
p_i^{fw}(\pi) & = \sum_{j=0}^{2i} (-1)^{j-i} c_j^{fw}(\pi) c_{2i-j}^{fw}(\pi), \\
e^{fw}(\pi) & = c_n^{fw}(\pi).
\end{align*}
%(cf. \cite[Corollary 15.5]{MS})
We remark that $e^{fw}(\pi)$ as defined here agrees with the fiberwise Euler class in the sense of \cite{HLLRW}, as follows from \cite[Theorem 5.6]{Prigge}. Let $B\aut(\tau_{\CP^n}^\RR)^{p,e}$ denote the classifying space for $\tau_{\CP^n}^\RR$-fibrations $(\pi,\zeta)$ with trivializations of the Pontryagin differences and the Euler difference,
$$pd_i(\pi,\zeta) = p_i^{fw}(\pi) - p_i(\zeta),\quad ed(\pi,\zeta) = e^{fw}(\pi) - e(\zeta),$$
and let $\Isom^+(\CP^n)$ denote the group of orientation preserving isometries of $\CP^n$.

\begin{theorem} \label{thm:orient}
The universal $\Isom^+(\CP^n)$-bundle with fiber $\CP^n$ has trivial Pontryagin and Euler differences, and the induced map
\begin{equation} \label{eq:isom}
B\Isom^+(\CP^n) \to B\aut(\tau_{\CP^n}^\RR)^{p,e}
\end{equation}
is a rational homotopy equivalence.
\end{theorem}

\begin{corollary}
A $\tau_{\CP^n}^{\RR}$-fibration is rationally equivalent to an $\Isom^+(\CP^n)$-bundle if and only if its Pontryagin and Euler differences are trivial.
\end{corollary}

\begin{theorem} \label{thm:taut orient}
Consider the ring $H^*\big(B\aut(\tau_{\CP^n}^\RR)^{p,e};\QQ\big)$ of characteristic classes of $\tau_{\CP^n}^\RR$-fibrations with trivialized Pontryagin and Euler differences.
\begin{enumerate}
\item For odd $n=2k+1$, the ring may be identified with the polynomial ring in the $n$ generators
$$\kappa_{ep_1},\ldots,\kappa_{ep_1^{k+1}},\kappa_{p_1^{k+2}},\ldots,\kappa_{p_1^{2k+1}}.$$

\item For even $n=2k$, the ring is a complete intersection of Krull dimension $n$ and embedding dimension $n + \binom{k}{2}$.
A minimal set of generators is given by
$$\kappa_{p_1^{k+1}},\ldots,\kappa_{p_1^{2k}},\quad \kappa_{p_1^{3k+1}},$$
and
$$\kappa_{p_1^{k+s-1}\beta_s},\ldots, \kappa_{p_1^{2k}\beta_s},$$
for $s=2,\ldots, k$,
where
$$\beta_s = (n+1)^s p_s - \binom{n+1}{s}p_1^s.$$
\end{enumerate}

In particular, all characteristic classes of $\tau_{\CP^n}^\RR$-fibrations with trivialized Pontryagin and Euler differences are tautological.
\end{theorem}

\begin{remark}
Theorem \ref{thm:taut orient} could also be read as a statement about the ring of characteristic classes of $\Isom^+(\CP^n)$-bundles with fiber $\CP^n$, in view of Theorem \ref{thm:orient}. The displayed generators are by no means canonical, there are other options.
\end{remark}

The above results have interesting consequences for the tautological ring of $\CP^n$. For an orientable $\tau_{\CP^n}^\RR$-fibration $(\pi,\zeta)$ over $B$, let $\pdideal \subseteq H^*(B;\QQ)$ denote the ideal generated by the coefficients $pd_{i|j}$, $ed_{|i}$, of the Pontryagin and Euler differences,
$$pd_i(\pi,\zeta) = \sum_{j=0}^{2i} pd_{i|j}\cdot \omega_{fw}(\pi)^j,\quad ed(\pi,\zeta) = \sum_{j=0}^{n} ed_{|j}\cdot \omega_{fw}(\pi)^j.$$
If $(\pi,\zeta)$ is not orientable, we define $\pdideal$ by pullback along $H^*(B;\QQ) \to H^*(B_\circ;\QQ)$, where $B_\circ\to B$ is the cover corresponding to the kernel of the action of $\pi_1(B)$ on the cohomology of the fiber. Let us call elements of $\pdideal \subseteq H^*(B;\QQ)$ `difference classes'.
We let $R^*(B) \subseteq H^*(B;\QQ)$ denote the subring of tautological classes.

\begin{theorem} \label{thm:taut}
Consider a $\tau_{\CP^n}^\RR$-fibration over a space $B$ whose classifying map fits in a homotopy commutative diagram
$$
\xymatrix{& B \ar[d] \\ B\Isom^+(\CP^n) \ar[r] \ar@{-->}[ur] & B\aut(\tau_{\CP^n}^\RR).}
$$
The induced ring homomorphism
\begin{equation} \label{eq:taut surj new}
R^*(B) \to H^*(B\Isom^+(\CP^n);\QQ)
\end{equation}
is surjective and the kernel is the ideal $R^*(B)\cap \pdideal$ of tautological difference classes.
\end{theorem}

For a $\CP^2$-fibration $\pi\colon E\to B$, the fiberwise Pontryagin and Euler classes are
\begin{align*}
p_1^{fw}(\pi) & = 3\omega_{fw}(\pi)^2 -2a_2(\pi)\cdot 1, \\
e^{fw}(\pi) & = 3\omega_{fw}(\pi)^2 + a_2(\pi)\cdot 1.
\end{align*}
A recent result of Baraglia \cite[Theorem 1.3(ii)]{Baraglia} implies that every smooth oriented $\CP^2$-bundle $\pi\colon E\to B$ has trivial Pontryagin and Euler differences,
\begin{equation} \label{eq:fw}
p_1(T_\pi E)  = p_1^{fw}(\pi), \quad e(T_\pi E) = e^{fw}(\pi).
\end{equation}
Baraglia uses \cite[Theorem 1.3(ii)]{Baraglia} as one of several ingredients in a computation of the tautological ring of $\CP^2$. Theorem \ref{thm:taut} applied to $B=B\Diff^+(\CP^2)$ gives a direct path from \eqref{eq:fw} to the computation that circumvents the other ingredients:
triviality of the Pontryagin and Euler differences means that $\pdideal = 0$ so
$$R^*(\CP^2) \to H^*(B\Isom^+(\CP^2);\QQ)$$
is an isomorphism. The target may be identified with $\QQ[\kappa_{p_1^2},\kappa_{p_1^4}]$ by Theorem \ref{thm:taut orient}.
 (Incidentally, this is the only even $n$ for which the target of \eqref{eq:taut surj new} is a polynomial ring.)
We remark that triviality of the Euler difference makes sense---and holds---for any oriented manifold bundle by \cite[\S3.2.1]{HLLRW}.

It is interesting to compare with Randal-Williams' computations \cite[\S4.4]{RW}; they inspire the following observations.
For a closed oriented smooth manifold $M$, the conditions on $\kappa_{\HL_i}$ that come from the family signature theorem give relations among the generators for $R^*(M)$. These relations are sometimes sufficient for determining $R^*(M)$ (as in Corollary \ref{cor:tautological ring}). The following shows they are not sufficient in general.

\begin{theorem} \label{thm:pdnz}
\begin{enumerate}
\item For every $\tau_{\CP^2}^\RR$-fibration
$$\CP^2 \to E \xrightarrow{\pi} B,\quad \zeta \to E,$$
such that
\begin{equation} \label{eq:necessary conditions}
e(\zeta) = e^{fw}(\pi),\quad \kappa_{\HL_i}(\pi,\zeta) = 0,\quad i>1,
\end{equation}
the difference class $pd_{1|0}$ is tautological and generates the ideal $R^*(B)\cap \pdideal$.

\item There exist $\tau_{\CP^2}^\RR$-fibrations satisfying \eqref{eq:necessary conditions} such that $pd_{1|0} \ne 0$. In particular, for such $\tau_{\CP^2}^\RR$-fibrations, $R^*(B) \to R^*(\CP^2)$ is not injective.
\end{enumerate}
\end{theorem}

\begin{remark}
The difference class $pd_{1|0}$ is directly related to the class $4\kappa_{p_1^2} - 7\kappa_{ep_1}$ featured in the computations of \cite[\S4.4]{RW}: for every $\tau_{\CP^2}^\RR$-fibration satisfying \eqref{eq:necessary conditions},
$$pd_{1|0} = \tfrac{1}{21}\big(4\kappa_{p_1^2} - 7\kappa_{ep_1}\big).$$
In particular, vanishing of the Pontryagin difference for smooth $\CP^2$-bundles explains why
\begin{equation} \label{eq:rcp2}
R^*(\CP^2) = \QQ[\kappa_{p_1^2},\kappa_{ep_1},\kappa_{p_1^4}]/(4\kappa_{p_1^2} - 7\kappa_{ep_1}),
\end{equation}
answering the question posed after Theorem D in \cite{RW}, but we stress that \eqref{eq:rcp2} should be viewed as a corollary of \cite[Theorem 1.3(ii)]{Baraglia}.
\end{remark}

\begin{remark}
For a closed oriented smooth manifold $M$, triviality of the Euler difference together with the conditions on $\kappa_{\HL_i}$ that come from the family signature theorem are necessary for being able to reduce the `structure group' of a $\tau_M$-fibration from $\aut(\tau_M)$ to $\Diff^+(M)$ rationally. For $S^m$ these conditions are sufficient (Corollary \ref{cor:sphere bundle}), but they are not sufficient in general. Indeed, the $\tau_{\CP^2}^\RR$-fibration in Theorem \ref{thm:pdnz}(2) has a non-trivial Pontryagin difference, so cannot be rationally equivalent to a smooth $\CP^2$-bundle.
\end{remark}

\begin{remark}
Theorem \ref{thm:taut} shows that the structure of the tautological ring $R^*(\CP^n)$ is intimately linked to the behavior of the Pontryagin differences of smooth oriented $\CP^n$-bundles.
For $n=2$, the Pontryagin difference vanishes, but this turns out to be a low-dimensional phenomenon. For $n$ sufficiently large, there are smooth oriented $\CP^n$-bundles with non-trivial Pontryagin differences. We thank Oscar Randal-Williams for suggesting a way to construct such bundles in a comment to an earlier version of this paper.

The rational homotopy groups of the space $\aut(\CP^n)/\tDiff(\CP^n)$, which classifies homotopically trivial block bundles with fiber $\CP^n$, can be computed using the surgery exact sequence (see e.g.~\cite[\S3]{BM1} for a review). After a few manipulations, this assumes the form of an exact sequence
$$0 \to \pi_k(\aut(\CP^n)/\tDiff(\CP^n))\tensor \QQ \to \bigoplus_{i\geq 1} H^{4i-k}(\CP^n;\QQ) \xrightarrow{\int_{\CP^n}} \QQ,$$
where the first map sends the equivalence class of a homotopically trivial block bundle $\pi\colon E\to S^k$, with fiber $\CP^n$ and stable fiberwise tangent bundle $\zeta$, to the sequence of cohomology classes $w^{-1}(\HL d_i(\pi,\zeta)) \in H^{4i-k}(\CP^n;\QQ)$, where $\HL d_i(\pi,\zeta)$ is the $i$th `$L$-class difference' and $w\colon H^{*-k}(\CP^n;\QQ) \to H^*(E;\QQ)$ comes from the Wang sequence associated to $\pi$. In particular, this shows the existence of block bundles over spaces within the rational homotopy type of $S^k$ with non-trivial $L$-differences and hence Pontryagin differences. To promote such block bundles to smooth bundles, \cite[Corollary D]{BuLa} implies that
$$\pi_k(\aut(\CP^n)/\Diff(\CP^n)) \tensor \QQ \to \pi_k(\aut(\CP^n)/\tDiff(\CP^n))\tensor \QQ$$
is surjective as long as $k$ is in the pseudoisotopy stable range for $\CP^n$, which holds if $2n \geq \max(2k+7,3k+4)$ by \cite{Igusa}.
\end{remark}

The applicability of Theorem 3.8 is not limited to the examples presented here.
Further computations and applications using Theorem \ref{thm:sullivan model} are worked out in the PhD thesis of Nils Prigge \cite{Prigge-thesis}.

The original motivation for this work was to understand the relation between the generalized Miller-Morita-Mumford classes and certain classes defined using graph complexes in the cohomology of the classifying space of the block diffeomorphism group of the manifold
$\#^g S^d\times S^d\setminus \interior D^{2d}$, see \cite{BM}. This application will be treated in a separate paper.

\begin{remark}
In \cite{Berglund-fibrations}, we constructed a different rational model for $B\aut_\circ(\xi)$, but this model is insufficient for the applications presented here, because it does not say anything about the universal $\xi$-fibration or the classifying map for the total bundle. Theorem \ref{thm:sullivan model}, on the other hand, does this. A precursor to Theorem \ref{thm:odd spheres} was obtained in \cite{Berglund-fibrations}, but it lacks an interpretation in terms of characteristic classes.
\end{remark}

{\bf Acknowledgements.}
We are grateful to Nils Prigge for his interest in this work and for numerous discussions.
We thank Oscar Randal-Williams for useful conversations and for bringing our attention to the account of the family signature theorem found in \cite[Appendix A]{KRW}.
The author was supported by the Swedish Research Council through grant no.~2015-03991.

\section{\texorpdfstring{$\xi$}{\textxi}-fibrations and characteristic classes}
In this section, we discuss the notion of a $\xi$-fibration and we give a homotopy theoretic model for the universal $\xi$-fibration. We also define generalized Miller-Morita-Mumford classes for $\xi$-fibrations, for bundles $\xi$ over Poincar\'e duality spaces.

\subsection{Bundles}
In this paper, the term \emph{bundle} (without further specification) can be taken to mean numerable fiber bundle (in the sense of Dold \cite[\S7]{Dold}) with fixed fiber $F$ and structure group $G$. 
Our main applications will be to vector bundles over CW-complexes and the reader can safely read the paper with this interpretation in mind without missing any of the main points. Nevertheless, certain arguments become clearer when expounded in greater generality and the more general results may be of independent interest.

By a \emph{bundle map} $\xi\to \xi'$ we understand a commutative diagram
$$
\xymatrix{E \ar[d] \ar[r]^-\varphi & E' \ar[d] \\ B \ar[r]^-f & B'}
$$
such that for every $b\in B$ the induced map on fibers $\varphi_b \colon E_b \to E_{f(b)}'$ is an isomorphism in the appropriate sense, as dictated by the structure group, cf.~\cite[\S2.5]{Steenrod}. Thus, for (oriented) vector bundles, we require $\varphi_b$ to be a linear (orientation preserving) isomorphism, as in \cite[p.26]{MS}, and so on.

\begin{remark}
More generally, the terms `bundle' and `bundle map' can be taken to mean `$\FF$-fibration' and `$\FF$-map', respectively, in the sense of May \cite{May}, for $\FF$ a category of fibers in the sense of \cite[Definition 4.1]{May} that satisfies the hypotheses of the classification theorem \cite[Theorem 9.2]{May}. The `structure group' $G$ will then mean the grouplike monoid of $\FF$-self-maps of the typical fiber $F$ as in \cite[Definition 4.3]{May}. Numerable fiber bundles are $\FF$-fibrations for a suitable choice of $\FF$, see Theorem 3.8 and Example 6.11 of \cite{May}. Another example: `bundle' could also mean `fibration with fiber weakly homotopy equivalent to $F$' for a fixed CW-complex $F$, in which case $G$ is the grouplike monoid $\aut(F)$ of homotopy automorphisms of $F$, and `bundle maps' are commutative diagrams as above with $\varphi_b$ a weak homotopy equivalence, cf.~\cite[Example 6.6]{May}.
\end{remark}

We let $\map(\xi,\xi')$ denote the space of bundle maps from $\xi$ to $\xi'$. Two key properties we need are the following:

\begin{itemize}
\item (Covering homotopy property) For bundles $\xi$ and $\xi'$, the forgetful map $\map(\xi,\xi') \to \map(B,B')$ is a fibration.

\item (Existence of a universal bundle) There is a bundle $\gamma$ such that $\map(\xi,\gamma)$ is weakly contractible for every bundle $\xi$.
\end{itemize}

\begin{remark}
The covering homotopy property as stated here is equivalent to the covering homotopy property for $\xi'$ in the definition of $\FF$-fibrations \cite[Definition 2.1]{May}. To see this, note that specifying a diagram
$$
\xymatrix{A\ar[d] \ar[r] & \map(\xi,\xi') \ar[d] \\
A\times I \ar[r]  \ar@{-->}[ur] & \map(B,B')}
$$
is tantamount to specifying a bundle map $A\times \xi \to \xi'$ and a homotopy $A\times B\times I \to B'$ of the base.
Existence of the lift is precisely the covering homotopy property for $\xi'$.
\end{remark}

\begin{remark}
The base space of the universal bundle $\gamma$ may be identified with $BG$ up to weak homotopy equivalence, and $\gamma$ is characterized by any of the following properties:
\begin{enumerate}
\item For CW-complexes $B$, the set of homotopy classes of maps $[B,BG]$ is in natural bijection with the set of equivalence classes of bundles over $B$.
\item The total space of the principal bundle associated to $\gamma$ is weakly contractible.
\item The space $\map(\xi,\gamma)$ is weakly contractible for every bundle $\xi$.
\end{enumerate}
The properties are equivalent, see \cite{BHP,Morgan}. Note that the last condition is equivalent to Theorem 3(4) of \cite{Morgan}. It is verified in the course of the proof of \cite[Theorem 9.2]{May} that (3) is satisfied, see the bottom of p.50.
\end{remark}

\subsection{\texorpdfstring{$\xi$}{\textxi}-fibrations}
Fix a bundle $\xi$ with projection $p\colon T \to X$ over a CW-complex $X$.

\begin{definition}
A \emph{$\xi$-fibration} over a space $B$ is a pair $(\pi,\zeta)$ consisting of
\begin{itemize}
 \item a fibration $\pi\colon E\to B$,
 \item a bundle $\zeta$ over the total space $E$,
\end{itemize}
such that $\zeta|_{E_b} \sim \xi$ for every $b\in B$.
We will refer to $\pi$ as the \emph{underlying fibration} and $\zeta$ as the \emph{total bundle} of the $\xi$-fibration $(\pi,\zeta)$.
\end{definition}
Here $\xi \sim \eta$ means that $\xi$ and $\eta$ are weakly equivalent, where a weak equivalence of bundles is a bundle map covering a weak homotopy equivalence. In particular, every fiber $E_b$ in a $\xi$-fibration is weakly equivalent to $X$, i.e., the underlying fibration is an `$X$-fibration'.

\subsection{The universal \texorpdfstring{$\xi$}{\textxi}-fibration}
We let $\aut(\xi)$ denote the topological monoid of bundle maps
$$
\xymatrix{T \ar[r]^-\varphi \ar[d]^-p & T \ar[d]^-p \\ X \ar[r]^-f & X}
$$
such that $f$ is a homotopy equivalence.

The classifying space $B\aut(\xi)$ classifies $\xi$-fibrations, in the sense that the set of homotopy classes of maps from a CW-complex $B$ to $B\aut(\xi)$ is in natural bijection with the set of equivalence classes of $\xi$-fibrations over $B$. This follows, indirectly, from Theorem 2.7 below together with \cite[\S11]{May}. See also the discussion in \cite[\S2]{Berglund-fibrations}. Alternatively, one can apply May's theory of $\FF$-fibrations, taking the category of fibers $\FF$ to be the category of bundles weakly equivalent to $\xi$ and letting the morphisms be the weak equivalences of bundles as defined above.

\begin{remark}
Some extremal examples:
\begin{itemize}
\item If $\xi$ is the trivial bundle over $X$ with fiber a point, then a $\xi$-fibration is the same thing as a fibration with fiber $X$, and $\aut(\xi) = \aut(X)$ is the grouplike topological monoid of self-homotopy equivalences of $X$.

\item If $X$ is a point and $\xi$ is the vector space $\RR^n$, viewed as a vector bundle over $X$, then a $\xi$-fibration over $B$ is the same thing as an $n$-dimensional real vector bundle over $B$, and $\aut(\xi) = \GL_n(\RR)$.
\end{itemize}
\end{remark}

The underlying fibration of the universal $\xi$-fibration may be identified with
$$\pi^{univ}\colon B\big(*,\aut(\xi),X\big) \to B\aut(\xi),$$
and its total bundle $\zeta^{univ}$ with
$$p_*\colon B\big(*,\aut(\xi),T\big) \to B\big(*,\aut(\xi),X\big).$$
In what follows, we will give another model for the universal $\xi$-fibration.
We will use the following two lemmas on homotopy orbit spaces.

For a grouplike monoid $\Gmonoid$ and a right $\Gmonoid$-space $X$, we use the geometric bar construction $B\big(X,\Gmonoid,*\big)$ of \cite[\S7]{May} as a model for the homotopy orbit space $X \dquot \Gmonoid$, though the results will of course be model independent. For $x_0\in X$, we let $\Gmonoid_{[x_0]} \subseteq \Gmonoid$ denote the stabilizer of $[x_0] \in \pi_0 X$, i.e.,
$$\Gmonoid_{[x_0]} = \set{g\in \Gmonoid}{[x_0 \cdot g] = [x_0]},$$
and we let $X_{x_0} \subseteq X$ denote the connected component containing $x_0$.

The following also appears as Lemma 4.10 in \cite{BM}. We repeat the short proof for completeness.

\begin{lemma} \label{lemma:homotopy orbits}
Consider a commutative square of the form
$$
\xymatrix{\Gmonoid \ar[d]^-\varphi \ar[r]^-{x_0 \cdot} & X \ar[d]^-f \\ \monoid \ar[r]^-{y_0\cdot } & Y}
$$
where $\Gmonoid$ and $\monoid$ are grouplike monoids, $\varphi$ is a map of topological monoids, $x_0 \in X$, $y_0\in Y$ and $f$ is a map of right $\Gmonoid$-spaces such that $f(x_0) = y_0$.
If the square is homotopy cartesian, then the induced map on components
$$\left( X \dquot \Gmonoid \right)_{x_0} \to \left( Y\dquot \monoid \right)_{y_0}$$
is a weak homotopy equivalence.
\end{lemma}

\begin{proof}
By \cite[Proposition 7.9]{May} the horizontal maps extend to homotopy fibration sequences
$$
\xymatrix{\Gmonoid \ar[d]^-\varphi \ar[r]^-{x_0 \cdot} & X \ar[d]^-f  \ar[r] & X \dquot \Gmonoid \ar[d] \\ \monoid \ar[r]^-{y_0\cdot } & Y \ar[r] & Y \dquot \monoid}
$$
The induced map on homotopy fibers, $\Omega_{x_0} \left( X \dquot \Gmonoid \right) \to \Omega_{y_0} \left( Y \dquot \monoid \right)$,
is a weak homotopy equivalence if the left square is assumed to be homotopy cartesian. Delooping this gives the desired result.
\end{proof}

\begin{lemma} \label{lemma:component}
The inclusion map
$$X_{x_0} \dquot \Gmonoid_{[x_0]} \to \left( X \dquot \Gmonoid \right)_{x_0}$$
is a weak homotopy equivalence.
\end{lemma}

\begin{proof}
Apply the previous lemma to the homotopy cartesian square
$$
\xymatrix{\Gmonoid_{[x_0]} \ar[r]^-{x_0 \cdot} \ar[d] & X_{x_0} \ar[d] \\ \Gmonoid \ar[r]^-{x_0 \cdot} & X,}
$$
where the vertical maps are the inclusions.
\end{proof}

Let $\map(X,BG)_\xi$ denote the component of a classifying map $f_\xi\colon X \to BG$ for $\xi$ and let $\aut(X)_{\xi}$ denote the stabilizer of the homotopy class of the classifying map, i.e., $\aut(X)_\xi = \aut(X)_{[f_\xi]}$ in the notation above. We remark that we may identify $\aut(X)_\xi$ with the monoid of homotopy automorphisms $f$ of $X$ such that $f^*(\xi)$ and $\xi$ are equivalent as bundles over $X$.

The following result gives a homotopy theoretical model for the universal $\xi$-fibration. It extends \cite[Corollary 2.4]{Berglund-fibrations} and \cite[Proposition 4.11]{BM}, which only identified the homotopy type of the base $B\aut(\xi)$. A key point is that the total bundle is classified by an evaluation map---this is what will allow us to write down explicit formulas for its characteristic classes in the relative Sullivan model later.

\begin{theorem} \label{thm:universal xi-fibration}
The underlying fibration $\pi^{univ}$ of the universal $\xi$-fibration is weakly equivalent to the map
$$
B\big(\map(X,BG)_\xi,\aut(X)_\xi,X\big) \to B\big(\map(X,BG)_\xi,\aut(X)_\xi,*\big),
$$
and the total bundle $\zeta^{univ}$ is classified by the map
$$ev\colon B\big(\map(X,BG)_\xi,\aut(X)_\xi,X\big) \to BG$$
induced by the evaluation map
$$\map(X,BG)_\xi \times X \to BG.$$
\end{theorem}

\begin{proof}
Fix a bundle map
$$
\xymatrix{T \ar[d]^-p \ar[r] & E_\infty \ar[d]^-{p_\infty} \\ X \ar[r] & BG}
$$
that classifies $\xi$. We may assume that the square is a pullback. Then, the diagram
$$
\xymatrix{\aut(\xi) \ar[d] \ar[r] & \map(\xi,\gamma) \ar[d] \\ \aut(X)_\xi \ar[r] & \map(X,BG)_\xi}
$$
is a pullback diagram. The vertical maps are fibrations by the covering homotopy property, so it is a homotopy pullback diagram. By Lemma \ref{lemma:homotopy orbits} the induced map
\begin{equation*} \label{eq:we}
B\big(\map(\xi,\gamma\big),\aut(\xi),*) \to B\big(\map(X,BG)_\xi,\aut(X)_\xi,*\big)
\end{equation*}
is a weak homotopy equivalence. Since $\map(\xi,\gamma)$ is contractible, the map
$$B\big(\map(\xi,\gamma),\aut(\xi),*\big) \to B\aut(\xi)$$
is a weak homotopy equivalence. This shows that the bottom horizontal maps in the following diagram are weak homotopy equivalences.
{\small $$
\xymatrix{B\big(*,\aut(\xi),X\big) \ar[d] & \ar[l]_-\sim B\big(\map(\xi,\univ),\aut(\xi),X\big) \ar[d] \ar[r]^-\sim & B\big( \map(X,BG)_{\xi},\aut(X)_\xi,X \big) \ar[d]\\
B\aut(\xi) & \ar[l]_-\sim B\big(\map(\xi,\gamma),\aut(\xi),* \big) \ar[r]^-\sim & B\big( \map(X,BG)_{\xi},\aut(X)_\xi,* \big) }
$$}
The squares are homotopy cartesian by Theorem 7.6 and Proposition 7.8 of \cite{May}, so it follows that the top horizontal maps are weak homotopy equivalences. This proves the first claim.

To prove the second claim, first note that we have an equivalence of bundles
$$
\xymatrix{B\big(\map(\xi,\univ),\aut(\xi),T\big) \ar[d] \ar[r]^-\sim & B\big(*,\aut(\xi),T\big) \ar[d] \\
B\big(\map(\xi,\univ),\aut(\xi),X\big) \ar[r]^-\sim & B\big(*,\aut(\xi),X\big). }
$$
Next, consider the diagram of bundle maps
$$
\xymatrix{B\big(\map(\xi,\gamma),\aut(\xi),T\big) \ar[d] \ar@{-->}[r] \ar@/^1.2pc/[rr]^-{\overline{ev}} & E \ar[d] \ar[r] & E_\infty \ar[d] \\ B\big(\map(\xi,\gamma),\aut(\xi),X\big) \ar[r]^-\sim & B\big(\map(X,BG)_\xi,\aut(X)_\xi,X\big) \ar[r]^-{ev} & BG,}
$$
where the right square is the pullback of the universal bundle along the evaluation map, the map $\overline{ev}$ is induced by the $\aut(\xi)$-equivariant map
$$
\map(\xi,\gamma)\times T \to E_\infty,\quad \big( (f,\varphi), t \big) \mapsto \varphi(t),
$$
and the dashed arrow exists by the universal property of the pullback.
This shows that the total bundle of the universal $\xi$-fibration is equivalent to the bundle over $B\big( \map(X,BG)_{f_\xi},\aut_\xi(X),X \big)$ classified by $ev$.
\end{proof}

\subsection{\texorpdfstring{$\xi$}{\textxi}-fibrations with prescribed holonomy}
Let $\monoid$ be a grouplike topological monoid acting effectively on $X$ by homotopy equivalences. That $\monoid$ acts effectively on $X$ means that $\monoid$ may be viewed as a submonoid of $\aut(X)$. We let $\aut_{\monoid}(\xi) \subseteq \aut(\xi)$ denote the submonoid of bundle maps $(f,\varphi)$ such that $f\in \monoid$. The classifying space $B\aut_{\monoid}(\xi)$ classifies $\xi$-fibrations $(\pi,\zeta)$ such that the holonomy action $\Omega B \to \aut(X)$ of the underlying fibration $X\to E \to B$ factors through $\monoid$. Examples of $\monoid$ that we have in mind are:
\begin{itemize}
\item $\aut_\circ(X)$ self-homotopy equivalences homotopic to the identity,
\item $\aut_+(X)$ orientation preserving self-equivalences (if $X$ is oriented),
\item $\aut_*(X)$ based homotopy equivalences (if $X$ is based),
\item $\aut_A(X)$ self-equivalences that fix a given subset $A\subseteq X$ pointwise.
\end{itemize}
If $\monoid = \aut_\circ(X)$, write $\aut_\circ(\xi)$ for $\aut_{\monoid}(\xi)$. Similarly for $\aut_+(X)$, $\aut_A(X)$, etc. If $A$ is a subspace of the fixed point set $X^{\monoid}$, we let $\aut_{\monoid}^A(\xi) \subseteq \aut_{\monoid}(\xi)$ denote the submonoid of those $(f,\varphi)$ for which $\varphi_x\colon T_x \to T_x$ is the identity for all $x\in A$.

Fix a classifying map $f_\xi\colon X\to BG$ for the bundle $\xi$.

\begin{theorem} \label{thm:decorated}
If the inclusion $A\to X$ is a cofibration, then the classifying space $B\aut_{\monoid}^A(\xi)$ is weakly homotopy equivalent to the homotopy orbit space
$$\map_A(X,BG)_\xi \dquot \monoid_{\xi},$$
where $\map_A(X,BG)$ denotes the space of maps $g\colon X\to BG$ such that $g|_A = (f_\xi)|_A$ and $\monoid_{\xi}$ denotes the stabilizer of the homotopy class of $f_\xi$ with respect to the induced action of $\monoid$ on $[X,BG]$.
\end{theorem}

\begin{proof}
Lemma \ref{lemma:homotopy orbits} applied to the homotopy cartesian square
$$
\xymatrix{\aut_{\monoid}^A(\xi) \ar[d] \ar[r] & \map_A(\xi,\gamma) \ar[d] \\ \monoid \ar[r] & \map_A(X,BG)}
$$
together with Lemma \ref{lemma:component} yield weak homotopy equivalences
$$\map_A(\xi,\gamma) \dquot \aut_{\monoid}^A(\xi) \xrightarrow{\sim} \left(\map_A(X,BG) \dquot {\monoid} \right)_{f_\xi} \xleftarrow{\sim} \map_A(X,BG)_\xi \dquot {\monoid}_{[\xi]},$$
The space $\map_A(\xi,\gamma)$ is contractible since it is the homotopy fiber of the restriction map $\map(\xi,\gamma) \to \map(\xi|_A,\gamma)$, so the map
$$\map_A(\xi,\gamma) \dquot \aut_{\monoid}^A(\xi) \to * \dquot \aut_{\monoid}^A(\xi)  = B\aut_{\monoid}^A(\xi).$$
is a weak homotopy equivalence.
\end{proof}

Here is an alternative characterization of the universal $\xi$-fibration.

\begin{proposition}
Fix a basepoint $x_0\in X$ and assume that the inclusion $\{x_0\} \to X$ is a cofibration. The underlying fibration of the universal $\xi$-fibration is weakly equivalent to the map
$$B\aut_*(\xi) \to B\aut(\xi),$$
and the total bundle over $B\aut_*(\xi)$ is classified by the map $B\aut_*(\xi) \to BG$ induced by the monoid map
$$\aut_*(\xi) \to  \aut(T_{x_0}) \simeq G$$
that sends $(f,\varphi)$ to the automorphism $\varphi_{x_0}\colon T_{x_0} \to T_{x_0}$.
\end{proposition}

\begin{proof}
By using the evaluation fibration, $\aut_*(\xi) \to \aut(\xi) \to X$, one shows that the homotopy fiber of $B\aut_*(\xi) \to B\aut(\xi)$ is weakly equivalent to $X$, and this is used to show that the map $B\aut_*(\xi) \to B\big(\aut(\xi),X\big)$ is a weak equivalence over $B\aut(\xi)$.
This proves the first claim. For the second claim, the diagram
$$\xymatrix{
B\big(\map(\xi,\univ),\aut(\xi),X\big) \ar[r]^-{ev} & BG \\
B\big(\map(\xi,\univ),\aut_*(\xi),*\big) \ar[u]_-\sim \ar[r] \ar[d]^-\sim & B\big(\map(\xi|_{x_0},\univ),\aut(\xi|_{x_0}),*\big) \ar[d]^-\sim \ar[u]_-\sim \\
B\aut_*(\xi) \ar[r] & B\aut(\xi|_{x_0})}$$
shows that the map $ev$, which classifies the total bundle, is weakly equivalent to the map $B\aut_*(\xi) \to BG$.
\end{proof}

\subsection{\texorpdfstring{$\xi$}{\textxi}-fibrations from manifold bundles}
Let $M$ be a smooth compact manifold of dimension $m$ and let $\tau_M = (p\colon TM \to M)$ denote its tangent bundle.
Consider a smooth $M$-bundle,
$$M\to E\xrightarrow{\pi} B,$$
i.e., a fiber bundle with fiber $M$ and structure group $\Diff(M)$.

Recall that the \emph{fiberwise tangent bundle}, or \emph{vertical tangent bundle}, $T_\pi E$ is a vector bundle over $E$ that may be defined as follows.
If $E$ and $B$ are smooth compact manifolds and $\pi$ is a surjective submersion, then $T_\pi E$ can be defined as the kernel of the differential $d\pi \colon TE \to \pi^*(TB)$. More generally, $T_\pi E$ may be defined as the vector bundle over $E$ with projection
$$P \times_{\Diff(M)} TM \to P \times_{\Diff(M)} M = E,$$
where $P\to B$ is the principal $\Diff(M)$-bundle associated to $\pi$. Thus, every smooth $M$-bundle has an underlying $\tau_M$-fibration with fibration $\pi$ and total bundle $T_\pi E$.

There is an evident map of monoids $d\colon \Diff(M) \to \aut(\tau_M)$ that sends a diffeomorphism $f\colon M\to M$ to its differential $df\colon \tau_M\to \tau_M$.

\begin{proposition} \label{prop:smooth bundles}
The $\tau_M$-fibration underlying the universal smooth $M$-bundle is classified by the map 
$B\Diff(M) \to B\aut(\tau_M)$ induced by the differential.
\end{proposition}

\begin{proof}
The space of embeddings $\Emb(M,\RR^\infty)$ is contractible and carries a free right action of $\Diff(M)$, so
$$\Emb(M,\RR^\infty) \to \Emb(M,\RR^\infty)/\Diff(M)$$
may be taken as a model for the universal principal $\Diff(M)$-bundle.
Hence, a model for the universal $M$-bundle is
$$M \to \Emb(M,\RR^\infty) \times_{\Diff(M)} M \to \Emb(M,\RR^\infty)/\Diff(M).$$
Given an embedding $f\colon M\to \RR^{m+k}$ and a point $x\in M$, the image $df_x(T_x M)$ is an $m$-dimensional linear subspace of $\RR^{m+k}$.
This defines a map into the Grassmannian, the generalized Gauss map,
$$
M \to \Gr_{m}(\RR^{m+k}).
$$
It is covered by a bundle map $\tau_M \to \gamma^m(\RR^{m+k})$, into the canonical $m$-dimensional vector bundle over $\Gr_{m}(\RR^{m+k})$, cf.~\cite[p.60--61]{MS}.
Varying the embedding, the Gauss maps give rise to a map
$$G\colon \Emb(M,\RR^{m+k}) \to \map\big(\tau_M,\gamma^m(\RR^{m+k})\big).$$
This map is $\Diff(M)$-equivariant and the action on the target factors through the differential $d\colon \Diff(M)\to \aut(\tau_M)$.
Letting $k\to \infty$, we obtain a bundle map
$$
\xymatrix{ B\big(\Emb(M,\RR^\infty),\Diff(M),TM \big) \ar[d]  \ar[rr]^{B(G,d,1)} && B\big(\map(\tau_M,\gamma^m),\aut(\tau_M), TM \big) \ar[d] \\
 B\big(\Emb(M,\RR^\infty),\Diff(M),M \big) \ar[rr]^{B(G,d,1)} && B\big(\map(\tau_M,\gamma^m),\aut(\tau_M), M \big),
}
$$
We recognize the left vertical map as a model for the fiberwise tangent bundle of the universal $M$-bundle and the right vertical map as a model for the total bundle of the universal $\tau_M$-fibration. This shows that the fiberwise tangent bundle of the universal $M$-bundle is pulled back from the total bundle of the universal $\tau_M$-fibration.
\end{proof}

\begin{remark}
In the classifying space interpretation, the map
$$\big[B,B\Diff(M)\big] \to \big[B,B\aut(\tau_M)\big]$$
may be identified with the forgetful map that sends the equivalence class of a smooth $M$-bundle over $B$ to the equivalence class of its underlying $\tau_M$-fibration.
\end{remark}

\subsection{Tautological classes for \texorpdfstring{$\xi$}{\textxi}-fibrations}
Let $X$ be a Poincar\'e duality space of formal dimension $m$ equipped with an orientation and consider a fibration
$$X\to E \xrightarrow{\pi} B,$$
such that $\pi_1(B)$ acts on $X$ by orientation preserving homotopy equivalences. Recall that the pushforward map, or integration along the fiber, $\pi_!\colon H^{k+m}(E)\to H^{k}(B)$ may be defined as the composite
$$H^{k+m}(E) \to E_\infty^{k,m} \to E_2^{k,m} = H^k(B,H^m(X)) \to H^k(B),$$
where the first two maps arise from the fact that the Serre spectral sequence of the fibration satisfies $E_2^{k,\ell} = 0$ for $\ell>m$, and the last map comes from the orientation. See e.g.~\cite{Grigoriev} and the references therein for a further discussion.

For a bundle $\xi$ over an oriented Poincar\'e duality space $X$, recall that $\aut_+(\xi)$ denotes the topological monoid of bundle maps from $\xi$ to itself that cover an orientation preserving homotopy equivalence of $X$.

\begin{remark}
If $M$ is an oriented manifold and if we fix our theory of bundles to be that of oriented vector bundles (i.e., we fix the fiber and structure group to be $F=\RR^m$ and $G=SO(m)$), then $\aut_+(\tau_M) = \aut(\tau_M)$, because a homotopy equivalence of $M$ that is covered by an orientation preserving bundle map of $\tau_M$ is automatically orientation preserving.
\end{remark}

\begin{definition}
Let $\xi$ be a bundle with structure group $G$ over an oriented Poincar\'e duality space $X$ of formal dimension $m$ and let
$$\pi\colon E\to B,\quad \zeta \to E,$$
be a $\xi$-fibration such that $\pi_1(B)$ acts on $X$ by orientation preserving homotopy equivalences.
For a cohomology class $c\in H^{k+m}(BG)$, we define
$$\kappa_c(\pi,\zeta) = \pi_!\big(c(\zeta)\big) \in H^{k}(B).$$
We will often denote $\kappa_c(\pi,\zeta)$ simply by $\kappa_c$ when there is no risk of confusion.
\end{definition}

The following is immediate from the definitions and Proposition \ref{prop:smooth bundles}, but it is a key observation and we record it as a theorem for reference.

\begin{theorem}
Let $M$ be a closed oriented smooth manifold of dimension $m$ and let $\vartheta$ be a smooth oriented $M$-bundle, i.e., a fiber bundle $\pi\colon E\to B$ with fiber $M$ and structure group $\Diff^+(M)$.

The generalized Miller-Morita-Mumford classes of $\vartheta$ agree with those of the underlying $\tau_M$-fibration, i.e.,
$$\kappa_c(\vartheta) = \kappa_c\big(\pi,T_\pi E\big) \in H^*(B),$$
for all $c\in H^{*+m}(BSO(m))$.

In particular, the universal classes $\kappa_c\in H^*(B\Diff^+(M))$ lift to $H^*(B\aut(\tau_M))$ under the map
$$H^*(B\aut(\tau_M))  \to H^*(B\Diff^+(M))$$
induced by the differential $\Diff^+(M) \to \aut(\tau_M)$.
\hfill $\square$
\end{theorem}

The tautological ring $R^*(M)$, in the sense of \cite{GGRW}, can be defined as the subring of the cohomology ring of $B\Diff^+(M)$ generated by the $\kappa$-classes. It is clear how to define an analog of the tautological ring for bundles.

\begin{definition}
For a bundle $\xi$ with structure group $G$ over an oriented Poincar\'e duality space $X$, we define $R^*(\xi)$ to be the subring of the cohomology ring of $B\aut_+(\xi)$ generated by the classes $\kappa_c$, for all $c\in H^*(BG)$.
\end{definition}

\begin{corollary}
For every closed oriented smooth manifold $M$, the differential gives rise to a surjective ring homomorphism
$R^*(\tau_M) \to R^*(M).$
\end{corollary}

\section{Rational homotopy theory of \texorpdfstring{$\xi$}{\textxi}-fibrations}
Let $\xi$ be a bundle over a simply connected finite CW-complex $X$ with structure group $G$ and let $\monoid$ be a connected topological monoid acting effectively on $X$ by homotopy equivalences. The aim of this section is to construct a relative Sullivan model, in the sense of rational homotopy theory (see e.g.~\cite[\S14]{FHT-RHT}), for the underlying fibration of the universal $\xi$-fibration with holonomy in $\monoid$,
$$X \to B\big(\aut_{\monoid}(\xi),X \big) \to B\aut_{\monoid}(\xi),$$
as well as a model for the classifying map $B\big(\aut_{\monoid}(\xi),X \big) \to BG$ of the total bundle.

\subsection{The Chevalley-Eilenberg cochain complex} \label{sec:CE}
Let $L$ be a differential graded Lie algebra over $\QQ$ with differential $\delta$.
For $n\in \ZZ$, we define $L\langle n \rangle$ by
$$
L\langle n \rangle_i = \left\{ \begin{array}{ll} L_i, & i>n, \\ \ker(L_n \xrightarrow{\delta} L_{n-1}), & i = n, \\ 0, & i<n. \end{array} \right.
$$
Recall that a Maurer-Cartan element is an element $\tau\in L_{-1}$ such that
$$\delta(\tau) + \frac{1}{2}[\tau,\tau] = 0.$$
If $\tau$ is a Maurer-Cartan element, then one can form the twisted dg Lie algebra $L^\tau$. It has the same underlying graded Lie algebra as $L$ but the differential is $\delta + [\tau,-]$.

The Chevalley-Eilenberg complex is the differential graded coalgebra
$$C_*(L) = \big( \Lambda sL,d = d_0 + d_1 \big),$$
where the differential is characterized by
\begin{align*}
d_0(sx) & = - s \delta(x), \\
d_1(sx \wedge sy) & = (-1)^{|x|} s[x,y].
\end{align*}
By definition, the Chevalley-Eilenberg cochain complex is the dual differential graded algebra $C^*(L) = C_*(L)^\vee$.

If $M$ is a differential graded left $L$-module, the Chevalley-Eilenberg complex with coefficients in $M$ is defined by
$$C^*(L,M) = \Hom(C_*(L),M).$$
The differential is the sum $\partial + t$, where
$$\partial(f) = d_M \circ f - (-1)^{|f|} f \circ d_{C_*(L)},$$
$$t(f) = \tau_L \cdot f.$$
Here $\tau_L \cdot f$ denotes the action of the universal twisting function $\tau_L \in \Hom(C_*(L),L)$ on $f\in \Hom(C_*(L),M)$. Explicitly,
$$t(f)(sx_1 \wedge \cdots \wedge sx_n) = \sum_{i=1}^n (-1)^{\epsilon_i} x_i \cdot f(sx_1\wedge \cdots \wedge \widehat{sx_i} \wedge \cdots \wedge sx_n),$$
$$\epsilon_i = |sx_i|(|f| + |sx_1| +\cdots + |sx_{i-1}|).$$

We call a cochain $f\in C^*(L,M)$ an \emph{$n$-cochain} if $f(sx_1\wedge \cdots \wedge sx_k) = 0$ unless $k=n$. Elements of $M$ may be identified with $0$-cochains. If $\alpha_i$ is a graded  vector space basis for $L$, then the \emph{dual $1$-cochains} $x_i\in C^*(L)$ are characterized by
$$x_i(s\alpha_j) = \delta_{ij}.$$
\begin{remark}
Our sign convention agrees with that of \cite{Tanre}, but differs from that of \cite{FHT-RHT}.
The signs are dictated by wanting the $1$-cochain $\tau_L\in C_*(L,L)$ defined by $\tau_L(sx) = x$ to be a twisting function in the sense of \cite{Quillen}.
\end{remark}

\subsection{Characteristic cochains of bundles}
Let $\xi$ be a bundle with structure group $G$ over $X$, classified by a map $f_\xi\colon X\to BG$.
Let $A$ be cdga model for $X$ and let $\Pi$ be a Lie model for $BG$ in the sense that $A$ and $C^*(\Pi)$ are quasi-isomorphic to $A_{PL}^*(X)$ and $A_{PL}^*(BG)$, respectively, as cdgas.
The rational homotopy class of $f_\xi$ is recorded by either
\begin{itemize}
\item the homotopy class of a morphism of cdgas
$$\varphi_\xi \colon C^*(\Pi) \to A,\quad \mbox{or}$$
\item the gauge equivalence class of a Maurer-Cartan element
$$\tau(\xi) \in A \ctensor \Pi,$$
in a certain completed tensor product.
\end{itemize}
For the latter, see e.g.~\cite[Theorem 1.5]{Berglund}. For $A$ of finite type, the completed tensor product can be taken to be $(A \ctensor \Pi)_n = \prod_{i} A^i \tensor \Pi_{i+n}$. If $A$ or $\Pi$ is finite dimensional, then $A\ctensor \Pi \cong A\tensor \Pi$. Concretely, if we fix a basis $\{q_i\}$ for $\Pi$ and we let $p_i\in C^*(\Pi)$ denote the dual $1$-cochains, then $\varphi_\xi$ and $\tau(\xi)$ are determined by certain cochains $p_i(\xi) \in A$, namely
\begin{equation*}
\varphi_\xi(p_i)  = p_i(\xi), \quad \tau(\xi) = \sum_i p_i(\xi) \tensor q_i.
\end{equation*}
We will refer to $p_i(\xi)$ as \emph{characteristic cochains} of the bundle $\xi$. The characteristic cochains are not unique, but the equivalence classes of $\varphi_\xi$ or $\tau(\xi)$ are.

In the special case when $G$ is connected and $H^*(BG;\QQ)$ is a free graded commutative algebra (e.g., if $G$ is a connected compact Lie group), the rational homotopy groups $\Pi = \pi_*(G)\tensor \QQ$, with trivial differential and Lie bracket, is a dg Lie algebra model for $BG$, and the cochains $p_i$ are cocycles. In this case, the homotopy class of $\varphi_\xi$, or the gauge equivalence class of $\tau(\xi)$, determines and is determined by the cohomology classes of the cocycles $p_i(\xi) \in A$.

\subsection{Lie models for monoid actions}
Let $\hl$ be a positively graded dg Lie algebra. Following \cite{Berglund-fibrations} we associate a simplicial group $\exp_\bullet(\hl)$ to $\hl$ as follows. In simplicial degree $n$ it is the nilpotent group associated to the nilpotent Lie algebra $Z_0(\Omega_n \tensor \hl)$ of $0$-cycles in the dg Lie algebra $\Omega_n \tensor \hl$, where $\Omega_n = A_{PL}^*(\Delta^n)$ is the cdga of polynomial differential forms on the standard $n$-simplex.

\begin{remark}
If $\hl$ is of finite type, then the simplicial group $\exp_\bullet(\hl)$ is isomorphic to the simplicial realization of the dg commutative Hopf algebra $\Gamma \hl = U \hl^\vee$, studied in \cite[\S25]{FHT-RHT}, see \cite[Proposition 3.8]{Berglund-fibrations}.
\end{remark}

Next, if $\hl$ acts on a cdga $\Lambda$ by derivations, then the simplicial group $\exp_\bullet(\hl)$ acts on the simplicial set $\langle \Lambda \rangle = \Hom_{cdga}(\Lambda,\Omega_\bullet)$. Indeed, the Lie algebra $Z_0(\Omega_\bullet \tensor \hl)$ acts on $\Omega_\bullet \tensor \Lambda$ by $\Omega_\bullet$-linear chain derivations. In each simplicial degree the action is nilpotent (since $\hl$ is assumed to be positively graded), so induces an action of the group $\exp Z_0(\Omega_\bullet\tensor \hl)$ on $\Omega_\bullet \tensor \Lambda$ by cdga automorphisms, and this induces an action on $\langle \Lambda \rangle \cong \Hom_{cdga(\Omega_\bullet)}(\Omega_\bullet \tensor \Lambda,\Omega_\bullet)$.

\begin{remark}
If $\hl$ is of finite type, then the action can alternatively be constructed as follows. That $\hl$ acts on $\Lambda$ by derivations means that the map
$$\alpha\colon \Lambda \to \Hom(U\hl,\Lambda),\quad \alpha(x)(\gamma) = x\cdot \gamma,$$
is a cdga morphism, where $\Hom(U\hl,\Lambda)$ is given the convolution product. If $\hl$ is of finite type, then  the natural cdga morphism
$$\beta\colon \Gamma \hl \tensor \Lambda \to \Hom(U\hl,\Lambda)$$
is an isomorphism. The composite map
$$\langle \Gamma \hl \rangle \times \langle \Lambda \rangle \cong \langle \Gamma\hl \tensor \Lambda \rangle \xrightarrow{\langle \beta^{-1} \alpha \rangle} \langle \Lambda\rangle$$
defines the group action.
\end{remark}

\begin{definition}
We will say that the action of $\hl$ on $\Lambda$ models the action of a topological monoid $\monoid$ on a space $X$ if $(\monoid,X)$ is rationally equivalent to $(|\exp_\bullet(\hl)|, |\Lambda|)$ in the category of pairs $(\Gmonoid,M)$ of topological monoids $\Gmonoid$ and $\Gmonoid$-spaces $M$, where morphisms $(g,m)\colon (\Gmonoid,M)\to (\Gmonoid',M')$ are pairs where $g\colon \Gmonoid \to \Gmonoid'$ is a map of topological monoids and $m\colon M\to M'$ is a map of $\Gmonoid$-spaces, and where $(g,m)$ is a rational equivalence if both $g$ and $m$ induce isomorphisms in rational homology.
\end{definition}

\begin{remark} \label{rmk:fibration}
Since the action of a grouplike monoid $\Gmonoid$ on a space $X$ can be recovered, up to homotopy, as the holonomy action of $\Omega B\Gmonoid$ on $X$ associated to the fibration $X\dquot \Gmonoid \to B\Gmonoid$, we have that $(\Gmonoid,X)$ is weakly equivalent to $(\monoid,Y)$ if and only if the associated fibrations $X\dquot \Gmonoid \to B\Gmonoid$ and $Y\dquot \monoid \to B\monoid$ are weakly equivalent.
\end{remark}

\begin{proposition}
Suppose that $\Lambda$ is a Sullivan algebra of finite type. If the action of $\hl$ on $\Lambda$ models the action of $\monoid$ on $X$, then the fibration
$$X \dquot \monoid \to B \monoid$$
is modeled by the relative Sullivan algebra
\begin{equation} \label{eq:relative Sullivan model}
C^*(\hl) \to C^*(\hl,\Lambda).
\end{equation}
\end{proposition}

\begin{proof}
The realization of the universal $U\hl$-coalgebra bundle
$$U \hl \to C_*(\hl,U\hl) \to C_*(\hl)$$
is a universal $\langle U\hl \rangle \cong \exp_\bullet(\hl)$-bundle, cf.~\cite[Theorem 3.9]{Berglund-fibrations} and \cite[\S25]{FHT-RHT}. It follows that the fibration $X\dquot \monoid \to B\monoid$ is rationally equivalent to
$$\langle C_*(\hl,U\hl) \rangle \times_{\langle U\hl \rangle} \langle \Lambda \rangle \to \langle C_*(\hl) \rangle.$$
The latter map isomorphic to the realization of the morphism of dg coalgebras $C_*(\hl,\Lambda^\vee) \to C_*(\hl)$, the dual of which is isomorphic to \eqref{eq:relative Sullivan model}.
\end{proof}

For a simply connected finite CW-complex $X$ with Sullivan model $\Lambda$, a well-known and widely used result is that $\Der \Lambda \langle 1 \rangle$, the positive truncation of the dg Lie algebra of derivations on $\Lambda$, is a Lie model for $B\aut_\circ(X)$. This is sketched in \cite[p.313]{Sullivan}. See also \cite{Lazarev,Prigge,SS,Tanre} for this and related results. We will here give a short proof that shows the slightly stronger statement that the action of $\aut_\circ(X)$ on $X$ is modeled by the action of $\Der \Lambda \langle 1 \rangle$ on $\Lambda$.

\begin{proposition} \label{prop:derivation}
For a simply connected finite CW-complex $X$ with Sullivan model $\Lambda$, the action of  $\Der \Lambda \langle 1 \rangle$ on $\Lambda$ models the action of $\aut_\circ(X)$ on $X$.
\end{proposition}

\begin{proof}
The action of $\hl = \Der \Lambda \langle 1 \rangle$ on $\Lambda$ gives rise to an action of the connected group $\exp_\bullet(\hl)$ on $\langle \Lambda \rangle$, which yields map of monoids
\begin{equation} \label{eq:derivations}
\exp_\bullet(\hl)\to \aut_\circ \langle \Lambda \rangle.
\end{equation}
The pairs $(\aut_\circ(X),X)$ and $(\aut_\circ \langle \Lambda \rangle, \langle \Lambda \rangle)$ are easily seen to be rationally equivalent (see, e.g., \cite[p.6]{Berglund-fibrations}), so we are done if we can show that \eqref{eq:derivations} is a weak homotopy equivalence. 
For $k\geq 1$, the map
$$H_k(\hl) \to \pi_k(\exp_\bullet(\hl))$$
that sends the homology class of a cycle $\theta\in \hl_k$ to the homotopy class of the $k$-simplex $\omega \tensor \theta \in Z_0(\Omega^*(\Delta^k) \tensor \hl)$, where $\omega$ is the fundamental form $k!dt_1\wedge \ldots \wedge dt_k$, is an isomorphism.
This can be checked directly by using the normalized chain complex for computing the homotopy groups of a simplicial vector space (the underlying simplicial set of $\exp_\bullet(\hl)$ is the simplicial vector space $Z_0(\Omega_\bullet \tensor \hl)$), or by using the isomorphism of simplicial sets $\exp_\bullet(\hl) \cong \MC_\bullet(s^{-1}\hl)$ (where the desuspension $s^{-1}\hl$ is viewed as an abelian dg Lie algebra) and invoking \cite[Theorem 4.6]{Berglund}.

On the other hand, it is well-known that the homotopy groups of the target of \eqref{eq:derivations} are computable in terms of derivations of $\Lambda$, cf.~\cite{BL,BuM,LS,Sullivan}. This goes as follows (cf.~\cite[Theorem 2.1]{LS}):
by the equivalence of homotopy categories between finite type Sullivan algebras and rational nilpotent spaces of finite $\QQ$-type,
$$\pi_k\aut_\circ \langle \Lambda \rangle = [S^k,\aut_\circ \langle \Lambda \rangle]_* \cong [S_\QQ^k,\aut_\circ \langle \Lambda \rangle]_* \cong [S_\QQ^k \times \langle \Lambda \rangle, \langle \Lambda \rangle]_{\langle \Lambda \rangle} \cong [\Lambda, H^*(S^k)\tensor \Lambda]^\Lambda.$$
The set $[\Lambda, H^*(S^k)\tensor \Lambda]^\Lambda$ of homotopy classes of cdga morphisms $f\colon \Lambda \to H^*(S^k)\tensor \Lambda$ over $\Lambda$ is in bijection with the set of homology classes of degree $k$ chain derivations $\theta\colon \Lambda\to \Lambda$ via $f(x) = 1\tensor x + z \tensor\theta(x)$, where $z\in H^k(S^k)$ is a generator.

The final step is to check that the composite map
$$H_k(\hl) \to \pi_k(\exp_\bullet(\hl)) \to \pi_k(\aut_\circ \langle \Lambda \rangle) \cong H_k(\Der \Lambda)$$
is the identity map. We leave this as an exercise to the reader.
\end{proof}

\subsection{A relative Sullivan model for the universal \texorpdfstring{$\xi$}{\textxi}-fibration}
Let $X$ be a simply connected finite complex, let $\xi$ be a bundle over $X$ with structure group $G$ and let $\monoid$ be a connected monoid acting on $X$ by homotopy equivalences. Assume that the space $BG$ is nilpotent and of finite $\QQ$-type.

Let $\Lambda$ be a Sullivan model for $X$ of finite type, let $\hl$ be a positively graded dg Lie algebra acting on $\Lambda$ by derivations, modeling the action of $\monoid$ on $X$ as in the previous section, and let $\Pi$ be a degreewise nilpotent finite type Lie model for $BG$. Fix characteristic cochains $p_i(\xi)\in \Lambda$ for the bundle $\xi$ and let
$$\tau(\xi) = \sum_i p_i(\xi) \tensor q_i$$
denote the corresponding Maurer-Cartan element in $\Lambda \ctensor \Pi$.

The dg Lie algebra $\hl$ acts by derivations on the dg Lie algebra $\Lambda \ctensor \Pi$ by $\theta \cdot (x\tensor q) = (\theta \cdot x)\tensor q$, so we may form the semi-direct product dg Lie algebra $\hl \ltimes \Lambda\ctensor \Pi$. The element $\tau(\xi)$ may be viewed as a Maurer-Cartan element in this semi-direct product. Define
$$\hl^\xi = \big(\hl \ltimes \Lambda \ctensor \Pi \big)^{\tau(\xi)}  \langle 0 \rangle,$$
i.e., the dg Lie algebra $\hl^\xi$ is obtained by twisting $\hl \ltimes \Lambda \ctensor \Pi$ by the Maurer-Cartan element $\tau(\xi)$ and then taking the non-negative truncation. We note that $\hl^\xi$ acts on $\Lambda$ by derivations via the evident map to $\hl$.

Next, define $1$-cochains $P_i\in C^*(\hl^\xi,\Lambda)$ by
$$P_i(s x\tensor q) = (-1)^{|x||q|} p_i(sq) x,$$
for $x\in \Lambda$ and $q\in \Pi$, and $P_i(s\theta) = 0$ for $\theta\in \hl$.

\begin{theorem} \label{thm:sullivan model}
\begin{enumerate}
\item The action of $\hl^\xi$ on $\Lambda$ models the action of $\aut_\monoid(\xi)$ on $X$. In particular, the universal $\xi$-fibration with holonomy in $\monoid$,
$$B\big(\aut_{\monoid}(\xi),X\big) \to B\aut_{\monoid}(\xi),$$
is modeled by the relative Sullivan algebra
$$C^*(\hl^\xi) \to C^*(\hl^\xi,\Lambda).$$

\item Characteristic cochains $p_i(\zeta)\in C^*(\hl^\xi,\Lambda)$ of the total bundle $\zeta$ of the universal $\xi$-fibration are given by
\begin{equation*}
p_i(\zeta) = p_i(\xi) + P_i,
\end{equation*}
where $p_i(\xi)\in \Lambda$ is viewed as a $0$-cochain in $C^*(\hl^\xi,\Lambda)$ and $P_i$ is the $1$-cochain described above.
\end{enumerate}
\end{theorem}

\begin{proof}
By Theorem \ref{thm:universal xi-fibration} and Theorem \ref{thm:decorated}, the fibration $B(\aut_\monoid(\xi),X) \to B\aut_\monoid(\xi)$ is weakly equivalent to
\begin{equation} \label{eq:universal fibration}
B\big( \map(X,BG)_\xi,\monoid,X\big) \to B\big( \map(X,BG)_\xi,\monoid,*\big).
\end{equation}
Let $C = \Lambda^\vee$ be the dual dg coalgebra and let $\homl$ denote the dg Lie algebra $\Hom(C,\Pi) \cong \Lambda \ctensor \Pi$.
An adaptation of the argument of \cite{Berglund-fibrations} (with $C$ replacing $\mathcal{C} L$ and $\hl$ replacing $(\Der L \ltimes_{\ad} sL)\langle 1 \rangle$, and using $\tau=\tau(\xi)$) shows that the map \eqref{eq:universal fibration} is modeled by the map of dg coalgebras
$$C_*\big(C_*(\homl^{\tau}\langle 0 \rangle),\hl,C\big) \to C_*\big(C_*(\homl^{\tau}\langle 0 \rangle),\hl\big).$$
Now observe that there is an isomorphism of dg coalgebras
$$C_*\big(C_*(\homl^{\tau}\langle 0 \rangle),\hl,M\big) \cong C_*\big( (\hl \ltimes \homl)^{\tau}\langle 0 \rangle,M\big),$$
natural in $U\hl$-module coalgebras $M$. It follows that the fibration $X\dquot \aut_\monoid(\xi) \to B\aut_\monoid(\xi)$ is rationally equivalent to the fibration constructed from the action of $\hl^\xi$ on $\Lambda$. In view of Remark \ref{rmk:fibration}, this means that the action of $\hl^\xi$ on $\Lambda$ models the action of $\aut_\monoid(\xi)$ on $X$.

To establish the formula for the characteristic cochains, it is by Theorem \ref{thm:universal xi-fibration} enough to understand how to model evaluation maps. This will be dealt with in the next section.
\end{proof}

\begin{remark} \label{rmk:explicit P_i}
For a more explicit description of the cochain $P_i$, choose a basis $\{x_\ell\}$ for $\Lambda$ and let $p_i^{x_\ell} \in C^*(\hl^\xi)$ denote $1$-cochain characterized by
$$p_i^{x_\ell}(sx_k\tensor q_j) = \delta_{(i,\ell),(j,k)}$$
and $p_i^{x_\ell}(s\theta) = 0$ for all $\theta\in \hl$. The cochain $P_i$ then assumes the form
\begin{equation} \label{eq:P_i}
P_i = \sum_{\ell} p_i^{x_\ell} x_\ell \in C^*(\hl^\xi,\Lambda),
\end{equation}
where we view $x_\ell\in \Lambda$ as a $0$-cochain in $C^*(\hl^\xi,\Lambda)$.
The sum is finite because of the truncation in the definition of $\hl^\xi$ and because $\Lambda$ is of finite type. Indeed, this implies that for $i$ fixed, the cochain $p_i^{x_{\ell}}$ is zero for all but finitely many $\ell$.
\end{remark}

\subsection{Rational models for evaluation maps}
\begin{theorem} \label{thm:evaluation map}
Let $X$ be a simply connected finite CW-complex, let $Z$ be a nilpotent space of finite $\QQ$-type and let $f\colon X\to Z$ be a map.
If $C$ is a fibrant cdgc model for $X$ of finite type, $\Pi$ is a degreewise nilpotent Lie model for $Z$ of finite type and $\tau\colon C\to \Pi$ is a twisting function that models the map $f\colon X\to Z$, then the map 
$$C_* \big(\Hom^\tau(C,\Pi)\langle 0 \rangle \big) \tensor C \to \Pi$$
determined by
$$1\tensor x \mapsto \tau(x),$$
$$sg \tensor x \mapsto g(x),$$
$$sg_1\wedge \cdots \wedge sg_n \tensor x \mapsto 0,\quad n>1,$$
is a twisting function that models the evaluation map
$$\map(X,Z)_f \times X \to Z.$$
\end{theorem}

\begin{proof}
First, we show that the map is a twisting function.
For a cdgc $\mathcal{C}$, a dg Lie algebra $\gl$ and a Maurer-Cartan element $\tau\in\MC(\gl)$, there is a bijection between sets of twisting functions,
$$\Tw(\mathcal{C},\gl^\tau) \xrightarrow{\cong} \Tw(\mathcal{C},\gl),\quad \rho \mapsto \rho + \tau \circ \epsilon,$$
where $\epsilon\colon \mathcal{C} \to \QQ$ is the counit.
If we apply this observation to $\gl = \Hom(C,\Pi)$ and $\mathcal{C} = C_*(\gl^\tau\langle 0 \rangle)$, the universal twisting function composed with the inclusion, $\pi \colon C_*(\gl^\tau\langle 0 \rangle) \to \gl^\tau\langle 0 \rangle \to \gl^\tau$, gives rise to a twisting function $\pi + \tau \circ \epsilon \colon \mathcal{C} \to \gl.$
Under the adjunction isomorphism,
$$\Tw(\mathcal{C}\tensor C,\Pi) \cong \Tw(\mathcal{C},\Hom(C,\Pi)),$$
this corresponds to the map in the statement of the theorem, showing it is a twisting function.

Next, we argue that the twisting function just constructed is a model for the evaluation map. We will go through the argument in the case when $\Pi$ is finite dimensional. For the general case, one writes $\Pi$ as an inverse limit of finite dimensional nilpotent dg Lie algebras $\Pi/\Pi\langle r \rangle$ and works with the inverse system, cf.~\cite[Remark 3.17]{Berglund-fibrations}. 

The argument of \cite[\S3.6]{Berglund-fibrations} goes through (upon replacing $\mathcal{C} L$ with $C$ and $(\Der L \ltimes_{\ad} sL)\langle 1 \rangle$ by $\hl$) and shows that the evaluation map
$$\map(X,Z_\QQ)\times X \to Z_\QQ$$
is modeled by the map
$$e\colon\MC_\bullet(\gl) \times \langle C \rangle \cong \MC\big( \Hom_{\Omega_\bullet}(C\tensor \Omega_\bullet,\Pi\tensor \Omega_\bullet) \big) \times \mathcal{G}(C\tensor \Omega_\bullet) \xrightarrow{ev} \MC_\bullet(\Pi),$$
where we write $\gl = \Hom(C,\Pi)$ as before. (Here, fibrancy of $C$ is needed to ensure that $\langle C \rangle$ has the correct homotopy type)

Next, the map $\lambda_\tau\colon \MC_\bullet(\gl^\tau\langle 0 \rangle) \to \MC_\bullet(\gl)$ that adds $\tau$ is a weak equivalence to the component that is determined by $\tau$, cf.~\cite[Corollary 4.11]{Berglund}. Also, note that the universal twisting function $\alpha\colon C_*(\gl^\tau\langle 0 \rangle)\to \gl^\tau\langle 0 \rangle$ induces an isomorphism $\langle C_*(\gl^\tau\langle 0 \rangle) \rangle \to \MC_\bullet(\gl^\tau\langle 0 \rangle)$. Now compose:
\begin{align*}
\langle C_*(\gl^\tau \langle 0 \rangle) \tensor C \rangle & \cong \langle C_*(\gl^\tau \langle 0 \rangle)  \rangle \times \langle C \rangle \\
& \xrightarrow{\alpha_*\times 1} \MC_\bullet(\gl^\tau \langle 0 \rangle) \times \langle C \rangle \\
& \xrightarrow{\lambda_\tau\times 1} \MC_\bullet(\gl) \times \langle C \rangle \\
& \xrightarrow{e}  \MC_\bullet(\Pi).
\end{align*}
By writing out definitions, one sees that this composite may be identified with the map induced by the twisting function in the statement of the theorem.
By the above, the composite map models the evaluation map $\map(X,Z_\QQ)_{rf} \times X\to Z_\QQ$, where $r\colon Z\to Z_\QQ$ is the rationalization. Since $X$ is a finite complex and $Z$ is nilpotent, this is rationally equivalent to $\map(X,Z)_f \times X \to Z$.
\end{proof}

\begin{remark}
Rational models for evaluation maps have been studied before by many authors, see e.g.~\cite{BuM-basic,Kuribayashi,LS,LS2}, but our approach of using twisting functions appears to be new.
\end{remark}

To finish the proof of Theorem \ref{thm:sullivan model}(2), note that since $\hl$ acts on $C = \Lambda^\vee$ by coderivations, the twisting function of Theorem \ref{thm:evaluation map} is $\hl$-equivariant (cf.~\cite[Proposition 3.18]{Berglund-fibrations}). This implies that it induces a twisting function
\begin{equation} \label{eq:twisting function}
C_*(C_*(\Hom^\tau(C,\Pi)\langle 0 \rangle),\hl,C) \to \Pi,
\end{equation}
given by the same formula, and that this models the evaluation map
$$B\big(\map(X,BG)_\xi,\monoid,X\big) \to BG.$$
Observing, as before, that
$$C_*(C_*(\Hom^\tau(C,\Pi)\langle 0 \rangle),\hl,C)^\vee \cong C^*(\hl^\xi,\Lambda),$$
it is an exercise to see that the Maurer-Cartan element in $C^*(\hl^\xi,\Lambda) \ctensor \Pi$ corresponding to \eqref{eq:twisting function} is given by the formula in Theorem \ref{thm:sullivan model}(2).

\section{Sample calculations and applications}

\subsection{Even spheres}
We consider the oriented tangent bundle $\xi = \tau_{S^m}$ of an even dimensional sphere $S^m$ for $m=2k \geq 2$ and compute a model for the universal $\xi$-fibration over $B\aut(\xi)$ using Theorem \ref{thm:sullivan model}.

Thus, $G = SO(2k)$ in this case and we have
$$H^*(BSO(2k);\QQ) = \QQ[p_1,\ldots,p_{k-1},e],\quad |p_i| =4i,\quad |e| = 2k,$$
where $p_i$ are the Pontryagin classes and $e$ the Euler class.
$$\Pi = \pi_*(SO(2k))\tensor \QQ =  \langle q_1,\ldots,q_{k-1},\epsilon \rangle,\quad |q_i|=4i-1,\quad |\epsilon| = 2k-1,$$
where $q_i$ is dual to $p_i$ and $\epsilon$ to $e$ under the Hurewicz pairing.

The minimal Sullivan model for $S^{2k}$ has the form
$$\left( \Lambda(x,y), x^2 \frac{\partial}{\partial y} \right),$$
with $x$ and $y$ in cohomological degrees $2k$ and $4k-1$, respectively.
The dg Lie algebra $\Der \Lambda \langle 1 \rangle$ has basis
$$\frac{\partial}{\partial x}, \quad \frac{\partial}{\partial y}, \quad x\frac{\partial}{\partial y},$$
and the only non-trivial differential is given by
$$\frac{\partial}{\partial x} \mapsto -2x\frac{\partial}{\partial y}.$$
Therefore, if we let $\hl\subset \Der \Lambda$ be the abelian dg Lie subalgebra spanned by $\frac{\partial}{\partial y}$, then the inclusion $\hl \to \Der \Lambda \langle 1\rangle$ is a quasi-isomorphism.

If we equip the cohomology $H = H^*(S^{2k};\QQ) = \QQ[x]/(x^2)$ with the trivial $\hl$-action, then the section $i\colon H\to \Lambda$ of $p$ sending $1$ to $1$ and the class of $x$ to $x$ is a quasi-isomorphism of dg $\hl$-modules (but not of algebras). The characteristic classes of $\xi$ are $p_i(\xi) = 0$ and $e(\xi) = 2x$ and as cocycle representatives in $\Lambda$ we may choose their images under $i$. The Maurer-Cartan element is $\tau(\xi) = 2x\tensor \epsilon$.

The section $i$ induces a quasi-isomorphism of dg Lie algebras
$$ \hl \times \big(H\tensor \Pi\big)\langle 0 \rangle \to \left(\hl \ltimes \Lambda\widehat{\tensor} \Pi \right)^{\tau(\xi)}\langle 0 \rangle.$$

Thus, a Lie model for $B\aut(\xi)$ is given by
$$\gl = \hl \times \big(H\tensor \Pi\big)\langle 0 \rangle$$
This is the abelian dg Lie algebra with trivial differential and basis
$$\frac{\partial}{\partial y},\quad 1\tensor q_1,\, \ldots,\, 1\tensor q_{k-1}, \quad 1\tensor \epsilon, \quad x\tensor q_{r}, \,\ldots,\, x\tensor q_{k-1},$$
where $r = \lceil \frac{k+1}{2} \rceil$. Thus,
$$R = C^*(\gl) = \QQ\left[a,p_1,\ldots,p_{k-1},e,p_r^x,\ldots,p_{k-1}^x\right],$$
where $a$ is the dual $1$-cochain of $\frac{\partial}{\partial y}$ and $p_i^x$ is dual to $x\tensor q_i$, and so forth. The differential is zero, so
$$H^*(B\aut(\tau_{S^{2k}});\QQ) \cong R.$$

The relative Sullivan model $C^*(\gl) \to C^*(\gl,\Lambda)$ may be identified with
$$R\to R[x,y],$$
where the differential is given by $dy = x^2 + a$. Clearly, the map
$$R[x,y] \to R[x]/(x^2+a).$$
is a quasi-isomorphism of dg $R$-algebras. Thus, the universal $\tau_{S^{2k}}$-fibration over $B\aut(\tau_{S^{2k}})$ is formal and in cohomology it is given by
$$R \to R[x]/(x^2 + a).$$

It follows from Theorem \ref{thm:sullivan model} that the characteristic classes of the total bundle $\zeta$ over $S^{2k}\dquot \aut(\xi)$ are given by
\begin{align*}
p_i(\zeta) & = p_i + p_i^x\cdot x, \\
e(\zeta) & = 2x + e,
\end{align*}
where $p_i^x$ should be interpreted as $0$ for $1\leq i \leq r-1$.

The pushforward map $\pi_!\colon R[x]/(x^2+a) \to R$ is determined by $R$-linearity and $\pi_!(x) = 1$.
The $\kappa$-classes can now be computed explicitly as elements of the polynomial ring $R$:
\begin{gather*}
\kappa_{p_i} = \pi_!(p_i(\zeta)) = \pi_!(p_i + p_i^x \cdot x) = p_i^x, \\
\kappa_{ep_i} = \pi_!((2x+e)\cdot (p_i + p_i^x \cdot x)) = 2p_i + e p_i^x, \\
\kappa_{e^2} = \pi_!( (2x+e)^2) = 4e, \\
\kappa_{e^3} = \pi_!( (2x+e)^3) = -8a+6e^2.
\end{gather*}
A look at the linear terms of these expressions shows that the classes
$$\kappa_{ep_1},\, \ldots,\, \kappa_{ep_k}, \quad \kappa_{p_r}\, ,\ldots,\, \kappa_{p_k},$$
where we write $p_k$ for $e^2$, are algebraically independent and generate $R$. This proves Theorem \ref{thm:even spheres}.

To prove Theorem \ref{thm:spheres} for $m$ even, observe that the sphere bundle associated to the universal oriented vector bundle may be identified the fibration
$$
S^m \to BSO(m) \xrightarrow{\pi} BSO(m+1)
$$
induced by the inclusion $SO(m) \to SO(m+1)$, and the fiberwise tangent bundle may be identified with the universal oriented vector bundle $\gamma^m$ over $BSO(m)$.
For this $\tau_{S^m}$-fibration, we have that
$$\kappa_{e\HL_i}= 2\HL_i \in H^{4i}(BSO(m+1);\QQ),$$
so it follows that $H^*(BSO(m+1);\QQ)$ is a polynomial ring in the classes
$$\kappa_{e\HL_1},\ldots,\kappa_{e\HL_k}.$$
The computation of $H^*(B\aut(\tau_{S^m})_L;\QQ)$ in the proof of Corollary \ref{cor:tautological ring} (which does not use Theorem \ref{thm:spheres}) then shows that $ BSO(m+1)\to B\aut(\tau_{S^m})_L$ induces an isomorphism in rational cohomology.

\subsection{Odd spheres}
Consider the sphere $S^m$, for $m=2k+1$ odd, and its tangent bundle $\xi = \tau_{S^m}$ viewed as an oriented vector bundle.
The structure group is $G = SO(2k+1)$ and the cohomology of its classifying spaces is $H^*(BSO(2k+1);\QQ) = \QQ[p_1,\ldots,p_k]$.

The minimal model for $S^{2k+1}$ has the form
$$\Lambda = \left(\Lambda x,0 \right),\quad |x|=2k+1.$$
The dg Lie algebra $\hl = \Der \Lambda$ is one-dimensional and spanned by
$$\frac{\partial}{\partial x}.$$
We have that $p_i(S^{2k+1}) = 0$ for all $i$, so we may take $\tau(\xi) = 0$.

The dg Lie algebra model $\hl^\xi$ for $B\aut(\tau_{S^{2k+1}})$ from Theorem \ref{thm:sullivan model} then has basis
$$
\frac{\partial}{\partial x},\quad q_1,\, \ldots,\, q_k,\quad  xq_r,\, \ldots,\, xq_k,
$$
where $r = \lceil \frac{k+1}{2} \rceil$. The differential is zero and the only non-trivial Lie brackets are
$$
\left[\frac{\partial}{\partial x}, xq_i \right] = q_i,
$$
for $i=r,\ldots,k$. Writing $z$, $p_i$ and $p_i^x$ for the dual $1$-cochains of $\frac{\partial}{\partial x}$, $q_i$ and $xq_i$, respectively, and setting $p_i^x = 0$ for $i<r$, we see that
$$R = C^*(\hl^\xi) = \Lambda\big(p_1,\ldots,p_k,p_r^x,\ldots,p_k^x,z\big),$$
where the non-trivial differentials are given by $dp_i= p_i^x z$. The model for the universal $\tau_{S^m}$-fibration then assumes the form
$$
R \to R[x],
$$
where $x$ is adjoined as an exterior generator and $dx = z$.

Let $\Omega$ be a graded commutative algebra equipped with a degree $-(2k+1)$ derivation $D$ such that $D^2 =0$.
We may adjoin a polynomial generator $z$ of degree $2k+2$ and form the cdga
$$\big( \Omega[z], zD \big).$$
There is a natural isomorphism of algebras
$$H^*(\Omega[z],zD) \cong Z(\Omega,D) \ltimes  \overline{H^*(\Omega,D)[z]},$$
where the right hand side denotes the ring of polynomials $\sum_i a_i z^i$ where $a_0$ a cycle in $(\Omega,D)$ and the coefficients $a_i$ for $i\geq 1$ are cohomology classes with respect to $D$.

We observe that $R$ may be identified with $\big(\Omega[z],zD\big)$ if we let
$$\Omega = \QQ[p_1,\ldots,p_k,p_r^x,\ldots,p_k^x],$$
equipped with the derivation $D$ defined by $D(p_i) = p_i^x$ and $D(p_i^x) =0$.
In other words,
$$\Omega = \Omega_{\QQ[p_1,\ldots,p_k]| \QQ[p_1,\ldots,p_{r-1}]}^*$$
is the cdga of K\"ahler differential forms on the polynomial ring $\QQ[p_1,\ldots,p_k]$ that are linear over the subring $\QQ[p_1,\ldots,p_{r-1}]$, with differential $D$ of degree $-(2k+1)$.

Moreover, $R[x]$ may be identified with the same construction applied to $\Omega[x]$ with derivation $D$ defined as above and extended by $D(x) =1$.

Observe that $H^*(\Omega[x],D) = 0$ since $xD + Dx = 1$. Furthermore, the projection $Z(\Omega[x],D) \to \Omega$ sending $a+bx$ to $a$ is an isomorphism with inverse $Dx$. Hence, we may identify the pullback and pushforward maps in cohomology
$$
\xymatrix{H^*(R) \ar[d]^-\cong \ar@<1ex>[r]^-{\pi^*} & H^*(R[x]) \ar@<1ex>[l]^-{\pi_!} \ar[d]^-\cong \\
Z(\Omega,D) \ltimes \overline{H^*(\Omega,D)[z]} \ar@<1ex>[r]^-i  &  \ar@<1ex>[l]^-D \Omega}
$$
where $i$ is the inclusion of $Z(\Omega,D)$ into $\Omega$.

For our particular $\Omega$, we obviously have
\begin{align*}
H^*(\Omega,D) & = \QQ[p_1,\ldots,p_{r-1}], \\
Z(\Omega,D) & = \QQ[p_1,\ldots,p_{r-1}] \ltimes D\Omega,
\end{align*}
so we can rewrite the cohomology of $(\Omega[z],zD)$ as
$$\QQ[p_1,\ldots,p_{r-1},z] \ltimes D \Omega_{\QQ[p_1,\ldots,p_k] | \QQ[p_1,\ldots,p_{r-1}]}^*,$$
where $z$ acts trivially on the second factor.

Since $p_i(\xi) = 0$ for all $i$, the characteristic cochains for the total bundle $\zeta$ are given by
$$p_i(\zeta) = P_i = p_i + p_i^x x \in R[x].$$
Observe that $p_i^x = D(p_i)$, so $p_i(\zeta) = p_i + D(p_i) x$. Since $D$ is a derivation, this implies that
$$c(\zeta) = c + D(c) x$$
for every $c\in \QQ[p_1,\ldots,p_k]$.
Hence,
$$\kappa_c = D(c).$$

The class represented by $z\in R$ may be identified with the Euler class of the underlying spherical fibration.

The fact that $p_i(\zeta) = p_i$ for $1\leq i \leq r-1$ allows us to identify the class represented by the cocycle $p_i\in R$ with the class we called $\alpha_{p_i}$ in the introduction.

Thus, Theorem \ref{thm:odd spheres} is proved.

To prove Theorem \ref{thm:spheres} in the case $m$ odd, use the Hirzebruch $L$-classes as generators for $H^*(BSO(m);\QQ)$ when constructing the model, so that
$$R= \Lambda\big( \HL_1,\ldots,\HL_k, \HL_r^x, \ldots, \HL_k^x,z \big),$$
with differential $d\HL_i = \HL_i^x z$ for $r\leq i \leq k$. We have
$$\kappa_{\HL_i} = \HL_i^x.$$
A model $R_L$ for the homotopy fiber $B\aut(\tau_{S^m})_L$ is then obtained by adding new generators $M_i$ to kill the cocycles $\HL_i^x$ for $i=r,\ldots,k$;
$$R_L = \Lambda\big( \HL_1,\ldots,\HL_k, \HL_r^x, \ldots, \HL_k^x,z,M_r,\ldots,M_k \big),$$
$$dM_i = \HL_i^x.$$
One checks that the cohomology of $R_L$ is a polynomial ring in the classes
$$\HL_1,\ldots,\HL_{r-1}, \quad \HL_r - M_r z, \ldots , \HL_k - M_k z,$$
and that the map
$$H^*(B\aut(\tau_{S^m})_L;\QQ) \to H^*(BSO(m+1);\QQ)$$
sends these to the Hirzebruch $L$-classes.

\subsection{Complex projective spaces}

\begin{lemma} \label{lemma:fw generator}
For every generator $\omega\in H^2(\CP^n;\QQ)$ there is a unique class $\omega_{fw}(\pi)\in H^2(E;\QQ)$ naturally associated to orientable $\CP^n$-fibrations $\pi\colon E\to B$ that fulfills $\omega_{fw}(\CP^n\to *) = \omega$. Here, naturality means that $g^*(\omega_{fw}(\pi)) = \omega_{fw}(\pi')$ whenever
$$
\xymatrix{E' \ar[d]^-{\pi'} \ar[r]^-g & E \ar[d]^-\pi \\ B' \ar[r]^-{f} & B}
$$
is a homotopy cartesian square in which $\pi$ and $\pi'$ are orientable $\CP^n$-fibrations.

The class $\omega_{fw}(\pi)$ is characterized by $\pi_!(\omega_{fw}(\pi)^{n+1}) = 0$ and $\omega_{fw}(\pi)|_{\CP^n} = \omega$.
\end{lemma}

\begin{remark}
The conditions $\pi_!(\omega_{fw}(\pi)^{n+1}) = 0$ and $\omega_{fw}(\pi)|_{\CP^n} = \omega$ mean that $\omega_{fw}(\pi)$ may be identified with the `coupling class'. That the coupling class admits a homotopy theoretical definition has been observed in \cite[Proposition 3.1]{KM}. Lemma \ref{lemma:fw generator} shows that the coupling class  is the only class naturally associated to orientable $\CP^n$-fibrations that restricts to $\omega$ in the fiber.
\end{remark}

\begin{proof}
Consider the universal orientable $\CP^n$-fibration,
$$\CP^n \to B\aut_{*,\circ}(\CP^n) \to B\aut_\circ(\CP^n).$$
Since $H^k(B\aut_\circ(\CP^n);\QQ) = 0$ for $k=2,3$, an application of the Serre spectral sequence shows that the restriction map
$$H^2(B\aut_{*,\circ}(\CP^n);\QQ) \to H^2(\CP^n;\QQ)$$
is an isomorphism. Define $\omega_{fw}$ to be the preimage of $\omega$ under this map.

Since the classifying map for an orientable $\CP^n$-fibration $\pi\colon E\to B$ factors through $B\aut_\circ(\CP^n)$, there is a homotopy cartesian square
$$
\xymatrix{E\ar[d]^\pi \ar[r]^-{g_\pi} & B\aut_{*,\circ}(\CP^n) \ar[d]^-{\pi^{univ}} \\ B \ar[r]^-{f_\pi} & B\aut_{\circ}(\CP^n),}
$$
which is uniquely determined up to homotopy by $\pi$.
Define $\omega_{fw}(\pi) = g_\pi^*(\omega_{fw})$.
We note that $\pi_!^{univ}(\omega_{fw}^{n+1}) = 0$ simply because $H^2(B\aut_\circ(\CP^n);\QQ) = 0$.
By naturality of the pushforward map, $\pi_!(\omega_{fw}(\pi)^{n+1}) = 0$. The property $\omega_{fw}(\pi)|_{\CP^n} = \omega$ holds because it holds in the universal orientable fibration (by definition).
\end{proof}

The classes 
$$1,\quad \omega_{fw}(\pi),\quad \omega_{fw}(\pi)^2,\quad \ldots,\quad \omega_{fw}(\pi)^n,$$
form a basis for $H^*(E;\QQ)$ as a $H^*(B;\QQ)$-module, and yield an isomorphism of $H^*(B;\QQ)$-modules
\begin{equation} \label{eq:splitting}
H^*(E;\QQ) \cong H^*(B;\QQ)\tensor H^*(\CP^n;\QQ),
\end{equation}
which is \emph{natural} in $\pi$. The basis of course depends on the choice of generator $\omega$, but the decomposition \eqref{eq:splitting} does not.
In what follows we fix the standard generator $\omega = -c_1(\gamma^1)$, which has the convenient property $\pi_!(\omega_{fw}(\pi)^n) = 1$ (cf.~\cite[p.170]{MS}).

We now proceed to prove Theorem \ref{thm:pu}.

The minimal Sullivan model for $\CP^n$ has the form
$$\Lambda = \big( \Lambda(x,y), d\big), \quad |x|=2,\,\, |y|=2n+1,$$
where $x$ is a cocycle that represents $\omega$ and $dy = x^{n+1}$. Let
$$H= H^*(\CP^n;\QQ) = \QQ[\omega]/(\omega^{n+1}).$$
The map $p \colon \Lambda \to H$ defined by $p(x) = \omega$ and $p(y) = 0$ is a quasi-isomorphism of cdgas. It admits a section $\iota\colon H\to \Lambda$, which is a quasi-isomorphism of cochain complexes (but which does not respect the multiplication).

Now consider a bundle $\xi$ over $\CP^n$, with structure group $G$ as in the statement of Theorem \ref{thm:pu}.
As a dg Lie model for $BG$ we may take the rational homotopy groups
$$\Pi = \pi_*(G) \tensor \QQ,$$
with trivial Lie bracket and differential. It has basis
$$q_1,\,\, q_2 \ldots,$$
where $q_i \in \pi_{*}(G) \cong \pi_{*+1}(BG)$ is dual to the generator $p_i\in H^*(BG;\QQ)$ under the Hurewicz pairing between cohomology and homotopy.

If we let $p_i(\xi)\in H$ denote the characteristic classes of $\xi$, then we may choose $\iota(p_i(\xi)) \in \Lambda$ as characteristic cochains of the bundle, so that
$$\tau(\xi) = \sum_i \iota(p_i(\xi)) \tensor q_i \in H\tensor \Pi.$$

One checks that the subspace $\al\subset \Der \Lambda \langle 1 \rangle$ spanned by the derivations
$$\frac{\partial}{\partial y}, \quad x \frac{\partial}{\partial y}, \quad \ldots, \quad x^{n-1} \frac{\partial}{\partial y},$$
forms an abelian dg Lie subalgebra with trivial differential, and that the inclusion $\al \to \Der \Lambda \langle 1\rangle$ is a quasi-isomorphism.
This means that the action of $\al$ on $\Lambda$ models the action of $\aut_\circ(\CP^n)$ on $\CP^n$.

If we equip the cohomology $H$ with trivial $\al$-action, then $\iota\colon H \to \Lambda$ is a quasi-isomorphism of dg $\al$-modules, and the induced map
$$\iota_*\colon \big( \al \ltimes H\tensor \Pi \big)^{\tau(\xi)}\langle 0 \rangle \to \big(\al \ltimes \Lambda \widehat{\tensor} \Pi \big)^{\iota_*(\tau(\xi))} \langle 0 \rangle,$$
is a quasi-isomorphism of dg Lie algebras ($\iota$ is not a morphism of algebras, but this does not matter since we tensor with the abelian dg Lie algebra $\Pi$). It follows from Theorem \ref{thm:sullivan model} that the dg Lie algebra
$$\gl = \big( \al \ltimes H\tensor \Pi \big)^{\tau(\xi)}\langle 0 \rangle = \al \oplus \big(H\tensor \Pi\big)\langle 0 \rangle$$
is a model for $B\aut_\circ(\xi)$.
Explicitly, the dg Lie algebra $\gl$ has trivial Lie bracket and trivial differential, and a basis is given by
$$
\begin{array}{lll}
x^k \frac{\partial}{\partial y}, & 0\leq k \leq n-1, \\
\omega^k \tensor q_i, & 0\leq k \leq n, \quad |q_i|\geq 2k.
\end{array}
$$

Since $\gl$ is concentrated in odd degrees, the Chevalley-Eilenberg construction $C^*(\gl)$ is a polynomial algebra with trivial differential,
$$
R = \QQ[v_0,\ldots,v_{k-1},p_{i,k}],
$$
generated by the $1$-cochains $v_k$ and $p_{i,k}$ dual to $x^k \frac{\partial}{\partial y}$ and $x^k \tensor q_i$, respectively.
It follows that we have an isomorphism of graded algebras
$$H^*(B\aut_\circ(\xi);\QQ) \cong R.$$
This is as far as we get from knowing a model for the base $B\aut_\circ(\xi)$. To identify which characteristic classes $v_k$ and $p_{i,k}$ represent, we must use the full force of Theorem \ref{thm:sullivan model}.

The relative Sullivan model $C^*(\gl) \to C^*(\gl,\Lambda)$ for the universal $\xi$-fibration from Theorem \ref{thm:sullivan model} may in the case at hand be identified with
$$R \to R \tensor \Lambda (x,y),$$
where the only non-trivial differential is given by
$$d(1\tensor y) = 1\tensor x^{n+1} + \sum_{\ell=0}^{n-1} v_{\ell} \tensor x^{\ell}.$$
The evident map $R \tensor \Lambda (x,y) \to R[x]/I$, where $I$ is the ideal generated by the element $x^{n+1} + \sum_{\ell=0}^{n-1} v_{\ell} x^{\ell}$ and where $R[x]/I$ is equipped with the zero differential, is a quasi-isomorphism of cdgas over $R$.
Thus, the universal $\xi$-fibration over $B\aut_{\circ}(\xi)$ is formal and modeled by the morphism of algebras
$$R \to R[x]/I.$$

The $R$-algebra $R[x]/I$ is free as an $R$-module with basis $1,x,\ldots,x^n$. The multiplication is determined by $R$-linearity and the relation
\begin{equation} \label{eq:relation}
x^{n+1} + \sum_{\ell=0}^{n-1} v_{\ell} x^{\ell} = 0.
\end{equation}
The pushforward map in cohomology is the degree $-2n$ map
$$\pi_! \colon R[x]/I \to R$$
determined by $R$-linearity and $\pi_!(x^n) = 1$ (note that $\pi_!(x^k) = 0$ for $k<n$ for degree reasons).
Since \eqref{eq:relation} contains no $x^n$-term, we have $\pi_!(x^{n+1}) = 0$, and by construction $x$ represents $\omega$ when restricted to the fiber, so the cocycle $x$ represents the class $\omega_{fw}$ by Lemma \ref{lemma:fw generator}.
This in turn implies that the class $v_i$ in the model represents the characteristic class $a_{n+1-i}$ featured in Theorem \ref{thm:pu}.

From Theorem \ref{thm:sullivan model} and Remark \ref{rmk:explicit P_i}, one sees that the characteristic classes of the total bundle $\zeta$ of the universal $\xi$-fibration over $B\aut_{\circ}(\xi)$ are given by
$$p_i(\zeta) = \iota(p_i(\xi)) + \sum_j p_{i,j} x^j.$$
This implies that the class $p_{i,j}$ in the model $R$ represents the characteristic class $p_{i|j}$. This finishes the proof of the first part of Theorem \ref{thm:pu}.

We now proceed to the second part of Theorem \ref{thm:pu}.
For a bundle $\xi$ over a space $X$, there is a homotopy fiber sequence
\begin{equation} \label{eq:cover}
B\aut_{\circ}(\xi) \to B\aut(\xi) \to B\Gamma(\xi),
\end{equation}
where
$$\Gamma(\xi) = \pi_0\aut(X)_{[\xi]} = \set{[f]\in \pi_0\aut(X)}{f^*\xi \cong \xi}$$
is the group of homotopy classes of self-homotopy equivalences of $X$ that fix the isomorphism class of the bundle $\xi$.

For $X=\CP^n$, the group $\pi_0 \aut(\CP^n)$ of homotopy classes of self-homotopy equivalences is cyclic of order two, generated by complex conjugation $c\colon \CP^n\to \CP^n$.
For a bundle $\xi$ over $\CP^n$, it follows that
$$\Gamma(\xi) \cong \left\{ \begin{array}{ll} \ZZ/2\ZZ, & c^*(\xi) \cong \xi, \\ 0, & c^*(\xi) \not\cong \xi. \end{array} \right.$$
For example, $\Gamma(\tau_{\CP^n}) = 0$ for the complex tangent bundle $\tau_{\CP^n}$, but for the underlying oriented vector bundle $\tau_{\CP^n}^\RR$ we have
$$\Gamma(\tau_{\CP^n}^\RR) \cong \left\{ \begin{array}{ll} \ZZ/2\ZZ, & \mbox{$n$ even}, \\ 0, & \mbox{$n$ odd}, \end{array} \right.$$
because complex conjugation is a diffeomorphism of $\CP^n$ which is orientation preserving if $n$ is even and orientation reversing if $n$ is odd.

Whenever $\Gamma(\xi)$ is finite, the rational Serre spectral sequence of the fibration \eqref{eq:cover}
collapses at the second page, and allows us to identify
$$H^*(B\aut(\xi);\QQ) = R^{\Gamma(\xi)},$$
where the latter denotes the invariant subring.
Complex conjugation is modeled by the involution on the minimal model $(\Lambda(x,y),d)$ given by
$$x\mapsto -x, \quad y\mapsto (-1)^{n+1}y.$$
Using this, one can work out that the action on $R$ is determined by
$$v_k\mapsto (-1)^{k+n+1} v_k, \quad p_{i,k} \mapsto (-1)^k p_{i,k}.$$
With that, the second part of Theorem \ref{thm:pu} is proved.

\begin{remark} \label{rmk:veronese}
If $H^*(BG;\QQ)$ is concentrated in degrees divisible by $4$, then the action agrees with the action given by multiplication by $(-1)$ in degrees $4\ell+2$ and by the trivial action in degrees $4\ell$.
Thus, if this is the case, and if $c^*(\xi) \cong \xi$, then the cohomology of $B\aut(\xi)$ may be identified with the subring $R^{(4)} \subset R$ of elements in degrees divisible by $4$. For example, this is the case for $\tau_{\CP^n}^\RR$ for $n$ even.
\end{remark}

The proof of Theorem \ref{thm:cd} will depend on an analysis of the universal $U(n+1)$-bundle with fiber $\CP^n$. This in turn may be identified with the projectivization of the universal complex vector bundle of rank $n+1$, so we proceed to discuss projectivizations.

Let $p\colon E\to B$ be a complex vector bundle of rank $n+1$. By taking fiberwise projectivization, we obtain a $U(n+1)$-bundle with fiber $\CP^n$,
\begin{equation} \label{eq:projectivization}
\CP^n \to \PP(E) \xrightarrow{\pi} B.
\end{equation}
There are two distinguished vector bundles over $\PP(E)$, the fiberwise canonical line bundle $L$ and the fiberwise tangent bundle $\zeta$. These make \eqref{eq:projectivization} into a $\gamma^1$-fibration and a $\tau_{\CP^n}$-fibration, respectively.

\begin{proposition}
The following equation relates the Chern classes of the $(n+1)$-dimensional complex vector bundle $p\colon E\to B$
and the Chern classes of the fiberwise tangent bundle $\zeta$ and the line bundle $L$ over the projectivization $\PP(E)$. For all $i$,
\begin{equation} \label{eq:chern classes}
c_i(\zeta) = \sum_{j=0}^i \binom{n+1-i+j}{j}c_{i-j}(E) \cdot c_1(\overline{L})^j,
\end{equation}
where $\overline{L}$ denotes the conjugate bundle.
\end{proposition}

\begin{proof}
The proof is essentially a fiberwise version of the classical computation of the Chern classes of $\CP^n$ \cite[Theorem 14.10]{MS}.
The pullback of the vector bundle $E$ along $\pi\colon \PP(E) \to B$ can be written as
$$\pi^*(E) = L\oplus L',$$
where $L'$ is the fiberwise orthogonal complement to $L$, and the fiberwise tangent bundle is given by
$$\zeta = \Hom(L,L').$$
Since $\Hom(L,L)$ may be identified with the trivial line bundle $\epsilon^1$, we have
\begin{align*}
\epsilon^1 \oplus \zeta & \cong \Hom(L,L) \oplus \Hom(L,L') \\
& \cong \Hom(L,L\oplus L') \\
& \cong \Hom(L,\pi^*(E)) \\
& \cong \overline{L} \tensor \pi^*(E).
\end{align*}
Since $c_i(\zeta) = c_i(\epsilon^1 \oplus \zeta)$, the formula \eqref{eq:chern classes} now follows from the well-known formula for the Chern classes of a tensor product of a line bundle and a vector bundle.
\end{proof}

\begin{remark}
The fiberwise tangent bundle $\zeta$ is $n$-dimensional, so $c_{n+1}(\zeta) = 0$. Hence, a special case of \eqref{eq:chern classes} is
\begin{equation} \label{eq:projective bundle formula}
\sum_{j=0}^{n+1} c_{n+1-j}(E) \cdot c_1(\overline{L})^j = 0.
\end{equation}
\end{remark}

By an application of Theorem \ref{thm:pu} to the canonical line bundle $\gamma^1$ over $\CP^n$, for which the structure group is $U(1)$ with $H^*(BU(1);\QQ)$ equal to the polynomial ring in $e=c_1$ and $c^*(\gamma^1) \not\cong \gamma^1$, we see that the ring of characteristic classes of $\gamma^1$-fibrations may be identified with
$$H^*(B\aut(\gamma^1);\QQ) = \QQ[a_2,\ldots,a_{n+1},e_{|0}].$$

\begin{proposition} \label{prop:a vs chern}
Consider a complex $(n+1)$-dimensional vector bundle,
$$p \colon E\to B,$$
and the $\gamma^1$-fibration formed by its projectivization,
$$\CP^n \to \mathbb{P}(E) \xrightarrow{\pi} B,\quad L \to \mathbb{P}(E).$$
The following equations express the Chern classes of the vector bundle $E$ in terms of the characteristic classes of the $\gamma^1$-fibration $(\pi,L)$, and vice versa.

For $i=1,\ldots,n+1$, we have that
\begin{equation} \label{eq:Chern classes}
c_i(E) = \sum_{j=0}^i \binom{n+1-i+j}{j} a_{i-j}(\pi) e_{|0}(\pi,L)^j.
\end{equation}

On the other hand
\begin{equation} \label{eq:first chern class}
e_{|0}(\pi,L) = \frac{1}{n+1}c_1(E),
\end{equation}
and
\begin{equation} \label{eq:a-classes}
a_i(\pi)  = \sum_{j=0}^i (-1)^j\binom{n+1-i+j}{j}c_{i-j}(E) \left(\frac{c_1(E)}{n+1}\right)^j,
\end{equation}
for $i=2,\ldots, n+1$.
\end{proposition}

\begin{proof}
By definition of the classes $e_{|j}$,
$$
e(L) = e_{|1}(\pi,L) \cdot \omega_{fw}(\pi) + e_{|0}(\pi,L) \cdot 1 \in H^2(\mathbb{P}(E);\QQ).
$$
Since both $\omega_{fw}(\pi)$ and $-e(L)$ restrict to $\omega = - c_1(\gamma^1)$ in the fiber, it follows that $e_{|1}(\pi,L) = - 1$.
Next, insert $\omega_{fw}(\pi) = -e(L) + e_{|0}(\pi,L) = e(\overline{L}) + e_{|0}(\pi,L)$ in the defining equation for the characteristic classes $a_i$,
$$\sum_{k=0}^{n+1} a_{n+1-k} \omega_{fw}(\pi)^k = 0,$$
compare with \eqref{eq:projective bundle formula} (remember that $c_1(\overline{L}) = e(\overline{L})$), and use that $1,e(\overline{L}),\ldots,e(\overline{L})^n$ form a $H^*(B;\QQ)$-module basis for the cohomology. This yields the formula \eqref{eq:Chern classes}. The formula \eqref{eq:a-classes} is derived in a similar fashion by inserting $e(\overline{L}) = \omega_{fw}(\pi) - e_{|0}(\pi,L)$ in \eqref{eq:projective bundle formula} and by observing that \eqref{eq:Chern classes} in particular shows that
\begin{equation*}
c_1(E) = (n+1)e_{|0}(\pi,L).
\end{equation*}
\end{proof}

\begin{corollary}
The map $BU(n+1)\to B\aut(\gamma^1)$ that classifies the $\gamma^1$-fibration obtained by projectivizing the universal complex vector bundle of rank $n+1$ is a rational homotopy equivalence.
\end{corollary}

\begin{proof}
The equations \eqref{eq:first chern class}, \eqref{eq:a-classes}, and \eqref{eq:Chern classes} give explicit formulas for the induced map in rational cohomology and its inverse.
\end{proof}

\begin{proposition}  \label{prop:ucd}
Every $U(n+1)$-bundle with fiber $\CP^n$ has trivial Chern differences.
\end{proposition}

\begin{proof}
Every $U(n+1)$-bundle with fiber $\CP^n$ is equivalent to the projectivization of some complex $(n+1)$-dimensional vector bundle $p\colon E\to B$
(e.g.~the complex vector bundle constructed from the associated principal $U(n+1)$-bundle), so we may as well assume the bundle is of the form
$$\CP^n \to \mathbb{P}(E) \xrightarrow{\pi} B.$$
Let $\zeta = T_\pi \PP(E)$ denote the fiberwise tangent bundle over $\mathbb{P}(E)$ and let $L$ be the canonical line bundle over $\PP(E)$.
By inserting
$$e(\overline{L}) = \omega_{fw}(\pi) - e_{|0}(\pi,L)$$
in the formula \eqref{eq:chern classes}, we obtain
\begin{align*}
c_i(\zeta) & = \sum_{j=0}^i \binom{n+1-i+j}{j} c_{i-j}(E) \left(\omega_{fw}(\pi) - e_{|0}(\pi,L)\right)^j \\
& = \sum_{j=0}^i \sum_{k=0}^j (-1)^{j-k}\binom{n+1-i+j}{j}\binom{j}{k} c_{i-j}(E) e_{|0}(\pi,L)^{j-k} \omega_{fw}(\pi)^k.
\end{align*}
Using the identity
$$
\binom{n+1-i+j}{j}\binom{j}{k} = \binom{n+1-i+k}{k} \binom{n+1-i+j}{j-k},
$$
and changing the order of summation, the above may be written as
$$
\sum_{k=0}^i \binom{n+1-i+k}{k} \left( \sum_{j=k}^i (-1)^{j-k} \binom{n+1-i+j}{j-k}c_{i-j}(E) e_{|0}(\pi,L)^{j-k} \right) \omega_{fw}(\pi)^k.
$$
By \eqref{eq:a-classes} and \eqref{eq:first chern class}, we recognize the inner sum as $a_{i-k}(\pi)$, so the above is equal to
$$c_i^{fw}(\pi) = \sum_{k=0}^i \binom{n+1-i+k}{k} a_{i-k}(\pi)\omega_{fw}(\pi)^k.$$
\end{proof}

Let $c=(c_i)$ be a sequence of cohomology classes $c_i \in H^{\ell_i}(X;\QQ)$. It is represented by a map
$$X\to \prod_i K(\QQ,\ell_i).$$
By interpreting the product of Eilenberg-Mac Lane spaces, call it $K$, as the classifying space $B\Omega K$, we can think of the sequence of cohomology classes $(c_i)$ as an equivalence class of a `bundle' $\xi_c$ over $X$ with `structure group' $\Omega K$, and the classes $c_i$ as the characteristic classes of $\xi_c$. A $\xi_c$-fibration
$$X\to E\xrightarrow{\pi} B,\quad \zeta \to E,$$
may be thought of as a fibration together with cohomology classes $c_i(\zeta)\in H^{\ell_i}(E;\QQ)$ that restricts to $c_i$ in the fiber.

Now let $c$ be the sequence of Chern classes of $\CP^n$.
We define the classifying space of $\tau_{\CP^n}$-fibrations with trivialized Chern differences as the homotopy pullback
\begin{equation} \label{eq:bc}
\begin{gathered}
\xymatrix{B\aut(\tau_{\CP^n})^c \ar[d] \ar[r] & B\aut_\circ(\CP^n) \ar[d]^-{c^{fw}} \\
B\aut(\tau_{\CP^n}) \ar[r]^-{c^{tot}} & B\aut(\xi_c),}
\end{gathered}
\end{equation}
where $c^{fw}$ arises by taking the fiberwise Chern classes of an orientable $\CP^n$-fibration and $c^{tot}$ by taking the Chern classes of the total bundle of a $\tau_{\CP^n}$-fibration.

\begin{theorem} \label{thm:tcd}
Every $PU(n+1)$-bundle with fiber $\CP^n$ has trivial Chern differences and the induced map $BPU(n+1) \to B\aut(\tau_{\CP^n})^c$ is a rational equivalence.
\end{theorem}

\begin{proof}
By Theorem \ref{thm:pu}, the ring of characteristic classes of $\tau_{\CP^n}$-fibrations is the polynomial ring
$$H^*(B\aut(\tau_{\CP^n});\QQ) = \QQ[a_2,\ldots,a_{n+1},c_{i|j}],$$
where we have one generator $c_{i|j}$ for each pair of integers $(i,j)$ such that $1\leq i \leq n$ and $0\leq j<i$.
By definition of the characteristic classes $c_{i|j}$, we have
$$c_i(\zeta) = \sum_{j=0}^n c_{i|j}(\pi,\zeta) \cdot \omega_{fw}(\pi)^j$$
for every $\tau_{\CP^n}$-fibration $(\pi,\zeta)$. By comparing coefficients, we see that the equation $c_i(\zeta) = c_i^{fw}(\pi)$ is equivalent to the equations
\begin{equation} \label{eq:kernel}
c_{i|j}(\pi,\zeta) = \binom{n+1-i+j}{j}a_{i-j}(\pi)
\end{equation}
for $0\leq j < i$.

It follows that the cohomology ring of $B\aut(\tau_{\CP^n})^c$ is isomorphic to
$$H^*(B\aut(\tau_{\CP^n})^c;\QQ) = \QQ[a_2,\ldots,a_{n+1},c_{i|j}]/I \cong \QQ[a_2,\ldots,a_{n+1}],$$
where $I$ is the ideal generated by the linear polynomials
$$c_{i|j} - \binom{n+1-i+j}{j}a_{i-j}$$
for $1\leq i \leq n$ and $0\leq j < i$.
We note in passing that this implies that the map $B\aut(\tau_{\CP^n})^c \to B\aut_\circ(\CP^n)$ is a rational equivalence.

By \eqref{eq:a-classes}, Proposition \ref{prop:ucd} and \eqref{eq:kernel}  applied to the universal $U(n+1)$-bundle with fiber $\CP^n$,
the map in rational cohomology induced by
$$q\colon BU(n+1) \to B\aut(\tau_{\CP^n})$$
may be identified with
$$
q^*\colon \QQ[a_2,\ldots,a_{n+1},c_{i|j}] \to \QQ[c_1,\ldots,c_{n+1}],
$$
where
\begin{gather*}
q^*(a_i) = \sum_{j=0}^i (-1)^j\binom{n+1-i+j}{j}c_{i-j} \left(\frac{c_1}{n+1}\right)^j, \\
q^*(c_{i|j}) = \binom{n+1-i+j}{j}q^*(a_{i-j}).
\end{gather*}
It follows that
$$\ker(q^*) = I.$$
Now, observe that $q$ factors as
$$\xymatrix{BU(n+1) \ar[dr]_-q \ar[r]^-f & BPU(n+1) \ar[d]^-r \\ & B\aut(\tau_{\CP^n}).}$$
The map $f$ is injective in rational cohomology, so
$$\ker(r^*) = \ker(f^*r^*) = \ker(q^*) = I.$$
This implies that the universal $PU(n+1)$-bundle with fiber $\CP^n$ has trivial Chern differences, and that the induced map
$$H^*(B\aut(\tau_{\CP^n})^c;\QQ) \to H^*(BPU(n+1);\QQ)$$
is injective. Since both source and target are abstractly isomorphic to a polynomial ring with generators in degrees $4,6,\ldots,2n+2$, the map must be an isomorphism.
\end{proof}

\begin{remark} \label{rmk:bpun}
The above shows in particular that each map in
$$BPU(n+1) \to B\aut(\tau_{\CP^n})^c \to B\aut_\circ(\CP^n)$$
is a rational equivalence. That $BPU(n+1) \to B\aut_{\circ}(\CP^n)$ is a rational equivalence has been observed before, cf.~\cite{Sasao,Kuribayashi,Prigge,KM}.
\end{remark}

In the calculations that follow, we will use the abbreviations $\omega_{fw} = \omega_{fw}(\pi)$, $a_i= a_i(\pi)$, $p_{i|j} = p_{i|j}(\pi,\zeta)$, etc, when there is no risk of confusion.
\begin{lemma} \label{lemma:pf}
For every orientable $\CP^n$-fibration $\pi\colon E \to B$ we have
\begin{gather}
\pi_!(\omega_{fw}^n) = 1, \\
\pi_!(\omega_{fw}^{n+k}) + a_k \in \aideal^2,
\end{gather}
for $1\leq k \leq n+1$, and
\begin{equation*}
\pi_!(\omega_{fw}^{n+k}) \in \aideal^2
\end{equation*}
for $k>n+1$.
Here, $\aideal \subseteq H^*(B;\QQ)$ is the ideal generated by $a_2,\ldots,a_{n+1}$.
\end{lemma}

\begin{proof}
Multiply both sides of the equation
$$\omega_{fw}^{n+1} + a_2 \cdot \omega_{fw}^{n-1} + \ldots + a_{n+1} = 0$$
with suitable powers of $\omega_{fw}$ and apply $\pi_!$. We omit the details.
\end{proof}

\begin{proof}[Proof of Theorem \ref{thm:extended tautological}]
Let $(\pi,\zeta)$ be the universal $\xi$-fibration over $B\aut_\circ(\xi)$.
By Lemma \ref{lemma:pf}, the class $\kappa_{\omega^{n+k}}$ is a polynomial in $a_2,\ldots,a_{n+1}$ with linear term $-a_k$.
This implies that the classes $\kappa_{\omega^{n+2}},\ldots, \kappa_{\omega^{2n+1}}$ are algebraically independent and that they generate the same subring as $a_2,\ldots,a_{n+1}$. Say $|p_i|=2r_i$. By definition of $p_{i|j}$, we have
$$p_i(\zeta) = \sum_{j=0}^n p_{i|j} \omega_{fw}^j,$$
where $p_{i|j} = 0$ if $j>r$ for degree reasons and $p_{i|r}$ is a rational number. By multiplying the above with $\omega_{fw}^{n-j}$ and applying $\pi_!$, we obtain the equality
$$\kappa_{\omega^{n-j} p_i} = p_{i|j} + p_{i|j+2} \kappa_{\omega^{n+2}} + \ldots + p_{i|r} \kappa_{\omega^{n-j+r}}.$$
This can be used to express $p_{i|j}$ in terms of the extended $\kappa$-classes by descending induction on $j=r-1,r-2,\ldots, 0$. The equality can also be used to show algebraic independence of the extended $\kappa$-classes by observing that the linear term of $\kappa_{\omega^{n-j} p_i}$ is $p_{i|j}$, modulo the subspace spanned by $a_2,\ldots, a_{n+1}$.
\end{proof}

We now turn to the proof of Theorem \ref{thm:taut orient}.

First, we observe that the fiberwise Pontryagin classes and the fiberwise Euler class can be defined for every $\CP^n$-fibration $\pi\colon E\to B$ such that $\pi_1(B)$ acts on the fiber by orientation preserving homotopy equivalences (this is weaker than ``orientability'' of the fibration for $n$ even, but equivalent to it for $n$ odd). The reason is that we may identify
$$H^*(\CP^n\dquot aut_+(\CP^n);\QQ) = H^*(\CP^n\dquot aut_\circ(\CP^n);\QQ)^{\ZZ/2\ZZ}$$
where $\ZZ/2\ZZ$ acts by $(-1)$ in degrees congruent to $2$ modulo $4$ (cf.~Remark \ref{rmk:veronese}). The classes $p_i^{fw}(\pi)$ and $e^{fw}(\pi)$ are invariant since their degrees are divisible by $4$ when $n$ is even.

We define the classifying space for $\tau_{\CP^n}^\RR$-fibrations with trivialized Pontryagin and Euler differences as a homotopy pullback
\begin{equation} \label{eq:bpe}
\begin{gathered}
\xymatrix{B\aut(\tau_{\CP^n}^\RR)^{p,e} \ar[r] \ar[d] & B\aut_+(\CP^n) \ar[d] \\ 
B\aut(\tau_{\CP^n}^\RR) \ar[r] & B\aut(\xi_{p,e})}
\end{gathered}
\end{equation}
similar to \eqref{eq:bc}, where $\xi_{p,e}$ is the cohomological bundle over $\CP^n$ associated to the Pontryagin classes and the Euler class of $\tau_{\CP^n}^\RR$.

The bottom horizontal map in \eqref{eq:bpe} is a rational equivalence, because by Theorem \ref{thm:universal xi-fibration} it may be identified with
$$\map(\CP^n,BSO(2n))_{\tau_{\CP^n}^\RR} \dquot \aut_+(\CP^n) \to \map(\CP^n,K_{p,e})_{p,e} \dquot \aut_+(\CP^n),$$
and the map
$$BSO(2n)\to K_{p,e} = \prod_{i=1}^{n-1} K(\QQ,4i) \times K(\QQ,2n)$$
recording the universal Pontryagin and Euler classes is a rational equivalence.
It follows that the top horizontal map in \eqref{eq:bpe} is a rational equivalence.

We proceed to make computations over $B\aut_+(\CP^n)$ using the fiberwise Pontryagin and Euler classes.
By Theorem \ref{thm:pu} applied to the orientation bundle,
$$H^*(B\aut_+(\CP^n);\QQ) = \QQ[a_2,\ldots,a_{n+1}]^{\Gamma},$$
where $\Gamma$ is trivial if $n$ is odd and cyclic of order $2$, acting by $a_k\mapsto (-1)^k a_k$, if $n$ is even.

\begin{lemma} \label{lemma:congruences}
Let $\pi\colon E\to B$ be a $\CP^n$-fibration such that $\pi_1(B)$ acts by orientation preserving homotopy equivalences on the fiber. Let $\aideal\subseteq H^*(B;\QQ)$ denote the ideal generated by $a_2,\ldots,a_{n+1}$. We have that
\begin{align*}
\kappa_e & = n+1, \\
\kappa_{ep_1} & = -4(n+1)a_2, \\
\kappa_{ep_1^\ell} & \equiv -2\ell(n+1)^\ell a_{2\ell} \pmod{\aideal^2}, && 4\leq 2\ell\leq n+1,\\
\kappa_{p_1^\ell} & \equiv -(n+1)^\ell a_{2\ell-n} \pmod{\aideal^2},&&  2 < 2\ell-n \leq n+1,\\
\kappa_{p_1^\ell} & = -(2n+3)(n+1)^{\ell-1} a_2, && 2\ell-n=2.
\end{align*}
\end{lemma}

\begin{proof}
We have that
\begin{align*}
p_1^{fw} & = (n+1)\omega_{fw}^2 - 2a_2,\\
e^{fw} & = (n+1)\omega_{fw}^n + (n-1)a_2\omega_{fw}^{n-2} + \ldots + 2a_{n-1}\omega_{fw} + a_n.
\end{align*}
It is a simple matter to expand $\left(p_1^{fw}\right)^\ell$ and $e^{fw}\left(p_1^{fw}\right)^\ell$ and then use Lemma \ref{lemma:pf} to calculate their pushforwards. We omit the details.
\end{proof}

Assume $n$ is odd, say $n=2k+1$. Lemma \ref{lemma:congruences} shows that, modulo $\aideal^2$ and up to multiplication by non-zero scalars, the classes
$$\kappa_{ep_1},\, \, \kappa_{ep_1^2},\ldots,\kappa_{ep_1^{k+1}}, \,\, \kappa_{p_1^{k+2}},\,\, \kappa_{p_1^{k+3}}, \ldots, \kappa_{p_1^{2k+1}},$$
agree with the classes
$$a_2,a_4,\ldots, a_{n+1}, \,\, a_3,a_5,\ldots, a_n.$$
This implies that the displayed $\kappa$-classes are algebraically independent and that they generate $H^*(B\aut_+(\CP^n);\QQ) = \QQ[a_2,\ldots,a_{n+1}]$.

Now assume $n$ is even, say $n=2k$. By what we said above, the ring
$$H^*(B\aut_+(\CP^n);\QQ)$$
may be identified with the subring of $\QQ[a_2,\ldots,a_{n+1}]$ consisting of all elements in degrees divisible by $4$.
This ring is generated by the following $n + \binom{k}{2}$ elements:
\begin{itemize}
\item $a_{\ell}$ for all even $\ell$, and
\item $a_{i,j} = a_i a_j$ for all odd $i$ and $j$ with $i\leq j$.
\end{itemize}
It is easy to see that the kernel of the map
$$\QQ[a_\ell,a_{i,j}] \to H^*(B\aut_+(\CP^n);\QQ)$$
is generated by the $\binom{k}{2}$ polynomials
$$a_{i,j}^2 - a_{i,i} a_{j,j},\quad i<j.$$
and these clearly form a regular sequence. Hence, the ring $H^*(B\aut_+(\CP^n);\QQ)$ is a complete intersection of Krull dimension $n$ and embedding dimension $n + \binom{k}{2}$.

We will now show that all classes are tautological. Clearly, it suffices to prove that the generators $a_\ell$ and $a_{i,j}$ are tautological.

For $n=2$, the computation is very easy. The expressions
\begin{align*}
\kappa_{p_1^2} & = -21a_2, \\
\kappa_{p_1^4} & = 81a_3^2 - 609 a_2^3,
\end{align*}
show that $\kappa_{p_1^2},\kappa_{p_1^4}$ generate the same subring as $a_2,a_3^2$.

The idea in the general case is the same but more care is required.
\begin{lemma}
Let $n$ be even, say $n=2k$.
Let $\beta_2,\ldots,\beta_{k+1}\in H^*(BSO(2n);\QQ)$ be classes of degree $|\beta_s| = 4s$ such that $\beta_{s|2s} =0$ for $s\leq k$ and $\beta_{s|1} \not\in \aideal^2$ for all $s$.
Then $H^*(B\aut_+(\CP^n);\QQ)$ is minimally generated by the classes
\begin{equation} \label{eq:generators}
\kappa_{p_1^{k+1}},\ldots,\kappa_{p_1^{2k}}, \quad \kappa_{p_1^{k+i} \beta_s},
\end{equation}
for $2\leq s \leq k+1$ and $s-1\leq i \leq k$.
\end{lemma}

\begin{proof}
Let $A=H^*(B\aut_\circ(\CP^n);\QQ) = \QQ[a_2,\ldots,a_{n+1}]$.
We have that
$$H^*(B\aut_+(\CP^n);\QQ) = A^{(4)}$$
is the subring of elements in degrees divisible by $4$. Let $R$ denote the subring of $A$ generated by the classes \eqref{eq:generators}.
Clearly, $R\subseteq A^{(4)}$. We will prove that $A^{4\ell} = R^{4\ell}$ for all $\ell$ by induction.

In degree $0$ there is nothing to prove.
In degree $4$, Lemma \ref{lemma:congruences} shows that
$$\kappa_{p_1^{k+1}} = -(2n+3)(n+1)^k a_2,$$
so $a_2\in R$, and hence $A^4 = R^4$.

Let $\ell>1$ and assume by induction that $A^{4\ell'} = R^{4\ell'}$ for all $\ell'<\ell$.
To show that $A^{4\ell} = R^{4\ell}$, it is enough to show that the generators in degree $4\ell$ belong to $R$. These are
$$a_i a_{2\ell-i}$$
for all odd $i$ such that $\max(3,2\ell-n-1)\leq i \leq \ell$ and, if $2\ell\leq n+1$,
$$a_{2\ell}.$$
Let us use the convention that $a_m=0$ for $m>n+1$. Lemma \ref{lemma:pf} shows that
$$\pi_!(a_i \omega_{fw}^{n+2\ell-i}) + a_ia_{2\ell - i} \in \aideal^3,$$
for all $i\leq \ell$
All elements of $\aideal^3$ of degree $4\ell$ belong to $R$ by induction. Hence,
\begin{equation}
\mbox{$a_i a_{2\ell-i}\in R\quad$ if and only if $\quad \pi_!(a_i \omega_{fw}^{n+2\ell-i})\in R$,}
\end{equation}
for all odd $i\leq \ell$.

We will now show that $a_ia_{2\ell-i} \in R$, or equivalently $\pi_!(a_i \omega_{fw}^{n+2\ell-i})\in R$, for all odd $i\leq \ell$ by induction on $i$, starting with the vacuous case $a_1a_{2\ell-1} = 0 \in R$. Thus, let $i$ be odd with $3\leq i \leq \ell$ and assume that $a_ja_{2\ell-j} \in R$ for all odd $j<i$. We may also assume $2\ell -n-1\leq i\leq n+1$, because otherwise $a_i=0$ or $a_{2\ell-i} = 0$ and there is nothing to prove. Say $i=2s-1$. We have $s\leq k+1$ because $i\leq n+1$. By the assumption $\beta_{s|2s} = 0$ for $s\leq k$ (and the fact that $A^2 = 0$) we may write
$$\beta_s =\sum_{j=2}^{2s} b_j \omega_{fw}^{2s-j},$$
where $b_j \in A^{2j}$. Multiplying by $\omega_{fw}^{2k+2\ell-2s}$ yields
\begin{equation} \label{eq:beta}
\beta_s \omega_{fw}^{2k+2\ell-2s} =\sum_{j=2}^{2s} b_j \omega_{fw}^{2k+2\ell -j}.
\end{equation}
Now note that since
\begin{equation} \label{eq:p_1}
\omega_{fw}^2 = \mu p_1^{fw} + \nu \kappa_{p_1^{k+1}}
\end{equation}
for non-zero rational numbers $\mu$, $\nu$, it follows that
$$\pi_!(\beta_s \omega_{fw}^{2k+2\ell-2s}) = c\kappa_{\beta_s p_1^{k+\ell-s}} + K,$$
where $c$ is a non-zero rational number and $K$ is a sum of products of $\kappa$-classes of lower degree, whence $K\in R$ by induction (this only uses that the $\kappa$-classes have degree divisible by four, not their specific form).
Also note that $\kappa_{\beta_s p_1^{k+\ell-s}}$ is one of the generators for $R$. Indeed,
$s-1\leq \ell-s$ holds because $2s-1=i \leq \ell$ by assumption, and  
$\ell-s \leq k$ holds because we have assumed $2\ell-i\leq n+1$.

Turning to the pushforward of the right hand side of \eqref{eq:beta}, consider
$$\pi_!(b_j\omega_{fw}^{2k+2\ell-j}) = b_j \pi_!(\omega_{fw}^{2k+2\ell-j}) .$$
The inequality $j\leq 2s=i+1\leq \ell+1$ implies $|b_j|= 2j <4\ell$ since $\ell>1$, so $b_j\in R$ for all even $j$. Also, for $j=2j'$ even and positive, the class $\pi_!(\omega_{fw}^{2k+2\ell-j})$ has degree $4(\ell-j') < 4\ell$, so it belongs to $R$ by induction. Thus, $\pi_!(b_j\omega_{fw}^{2k+2\ell-j}) \in R$ for all even $j$.

Now for odd $j$, we may write
$$b_j = \sum_{\substack{q=3 \\ q \,\, \textrm{odd}}}^j f_{j-q}a_q$$
where $f_{j-q} \in A^{2(j-q)}$. Hence,
\begin{equation} \label{eq:bj}
\pi_!(b_j\omega_{fw}^{2k+2\ell-j}) = \sum_{\substack{q=3 \\ q \,\, \textrm{odd}}}^j f_{j-q} \pi_!(a_q \omega_{fw}^{2k+2\ell-j}).
\end{equation}
For $q<j$, both factors $f_{j-q}$ and $\pi_!(a_q\omega_{fw}^{2k+2\ell-j})$ have degrees that are smaller than $4\ell$ and divisible by $4$, so they belong to $R$ by induction. For $q=j$, we have $f_0 \in \QQ$ and $\pi_!(a_j\omega_{fw}^{2k+2\ell-j})\in R$ by the inductive hypothesis that $a_ja_{2\ell-j} \in R$ for odd $j<i$.

Thus, we have shown that $\pi_!(b_j\omega_{fw}^{2k+2\ell-j}) \in R$ for all $j<i$ and that, for $j=i$, all terms in the right hand side of \eqref{eq:bj}, except possibly the one corresponding to $q=i$, belong to $R$.
But applying $\pi_!$ to \eqref{eq:beta} and using our above observation that $\pi_!(\beta_s \omega_{fw}^{2k+2\ell-2s}) \in R$, we can conclude that also this last term,
$$f_0\pi_!(a_i \omega_{fw}^{2k+2\ell-i}),$$
belongs to $R$. The assumption that $\beta_{s|1} \not\in \aideal^2$ means that $b_i$ must contain a term of the form $f_0 a_i$ with $f_0$ a non-zero rational number. Therefore, $\pi_!(a_i \omega_{fw}^{2k+2\ell-i})$ belongs to $R$ as well, and this finishes the induction on $i$.

To show that $a_{2\ell}\in R$, we use Lemma \ref{lemma:congruences}. It shows that
$$\kappa_{p_1^{k+\ell}} + (n+1)^{k+\ell} a_{2\ell} \in \aideal^2.$$
All elements of $\aideal^2$ in degree $4\ell$ belong to $R$ by induction or by the now proved statement that $a_ia_{2\ell -i} \in R$ for all odd $i$.
Also, $\kappa_{p_1^{k+\ell}}$ is one of the generators for $R$ since we may assume $\ell\leq k$ (otherwise $a_{2\ell} =0$ and there is nothing to prove). This finishes the induction on $\ell$ and concludes the proof.
\end{proof}

We have that
$$p_1^{fw} = (n+1) \omega_{fw}^2 -2a_2,$$
$$p_2^{fw} = \binom{n+1}{2} \omega_{fw}^4 -(2n-4)a_2\omega_{fw}^2-6a_3\omega_{fw} + 2a_4 + a_2^2.$$
We do not need a general formula for $p_s^{fw}$, but we record the following properties when $2s\leq n$: the leading term is
$$p_s^{fw} = \binom{n+1}{s} \omega^{2s} + \ldots,$$
and
$$p_s^{fw} \equiv (-1)^{s-1} (4s-2)a_{2s-1}\omega_{fw} + (-1)^s 2a_{2s} \pmod{\aideal^2 +(\omega_{fw})^2}.$$
These facts imply that we may use
$$\beta_s = (n+1)^sp_s^{fw} -\binom{n+1}{s} \left(p_1^{fw}\right)^s$$
for $s=2,3\ldots,k$. Finally, we have that
$$\left(p_1^{fw}\right)^{k+1} = -(2n+3)(n+1)^{k}a_2\omega_{fw}^n + \ldots  - (n+1)^{k+1}a_{n+1}\omega_{fw} + (-2a_2)^{k+1}.$$
This implies that we can use $\beta_{k+1} = \left(p_1^{fw}\right)^{k+1}$. By that, Theorem \ref{thm:taut orient} is proved.

\begin{proof}[Proof of Theorem \ref{thm:orient}]
Vanishing of the Pontryagin differences and the Euler difference for $\Isom_\circ(\CP^n)$-bundles follows from Corollary \ref{cor:pun}, because $\Isom_\circ(\CP^n) \cong PU(n+1)$. One can show that the map $\Isom_\circ(\CP^n)\cong PU(n+1) \to B\aut_\circ(\tau_{\CP^n}^\RR)^{p,e}$ is a rational equivalence as in the proof of Theorem \ref{thm:tcd}. For $n$ odd, $\Isom^+(\CP^n) = \Isom_\circ(\CP^n)$ and $B\aut(\tau_{\CP^n}^\RR)^{p,e} = B\aut_\circ(\tau_{\CP^n}^\RR)^{p,e}$ so nothing more needs to be said. For $n$ even, one uses that $\Isom^+(\CP^n) \cong \Isom_\circ(\CP^n)\rtimes \ZZ/2\ZZ$, and similarly for $B\aut(\tau_{\CP^n}^\RR)^{p,e}$.
\end{proof}

\begin{proof}[Proof of Theorem \ref{thm:taut}]
Theorem \ref{thm:orient} and Theorem \ref{thm:taut orient} show that all characteristic classes of $\Isom^+(\CP^n)$-bundles with fiber $\CP^n$ are tautological, so the map \eqref{eq:taut surj new} is surjective.

As in the proof of Theorem \ref{thm:tcd}, the map in cohomology induced by
$$B\Isom_\circ(\CP^n) \sim_\QQ B\aut_\circ(\tau_{\CP^n}^\RR)^{p,e} \to B\aut_\circ(\tau_{\CP^n}^\RR)$$
may be identified with the homomorphism
$$\QQ[a_2,\ldots,a_{n+1},p_{i|j},e_{|i}] \to \QQ[a_2,\ldots,a_{n+1},p_{i|j},e_{|i}]/\pdideal^{univ},$$
where $\pdideal^{univ}$ is the ideal generated by the coefficients of the Pontryagin and Euler differences of the universal orientable $\tau_{\CP^n}^\RR$-fibration. It follows that the map
\begin{equation} \label{eq:comp}
R^*(\tau_{\CP^n}^\RR) \to H^*(B\Isom^+(\CP^n);\QQ)
\end{equation}
has kernel $R^*(\tau_{\CP^n}^\RR)\cap \pdideal^{univ}$.
The hypotheses on the given $\tau_{\CP^n}^\RR$-fibration over $B$ imply a factorization of \eqref{eq:comp} into surjective ring homomorphisms
$$R^*(\tau_{\CP^n}^\RR) \xrightarrow{g} R^*(B) \xrightarrow{f} H^*(B\Isom^+(\CP^n);\QQ).$$
This implies that $\ker(f) = g(\ker(fg))$, which is seen to be equal to $R^*(B)\cap \pdideal$ by the above.
\end{proof}

\begin{proof}[Proof of Theorem \ref{thm:pdnz}]
The ring $H^*(B\aut(\tau_{\CP^2}^\RR)^e;\QQ)$ may be identified with the subring of $\QQ[a_2,a_3,p_{1|0},p_{1|1}]$ generated by the elements
\begin{equation} \label{eq:gens}
a_2,\quad p_{1|0},\quad p_{1|1}^2,\quad a_3p_{1|1},\quad a_3^2.
\end{equation}
If $(\pi,\zeta)$ denotes the universal $\tau_{\CP^2}^\RR$-fibration with trivialized Euler difference, then
\begin{align*}
p_1(\zeta) & = 3\omega_{fw}^2+p_{1|1}\omega_{fw} + p_{1|0}, \\
p_1^{fw}(\pi) & = 3\omega_{fw}^2-2a_2, \\
e(\zeta) = e^{fw}(\pi) & = 3\omega_{fw}^2 + a_2.
\end{align*}
With the above expressions at hand, it is straightforward to compute the following.
\begin{align*}
\kappa_{p_1^2} & = -9a_2 + 6p_{1|0} + p_{1|1}^2, \\
\kappa_{\HL_2} & = -\tfrac{4}{15}a_2 - \tfrac{2}{15} p_{1|0} - \tfrac{1}{45} p_{1|1}^2, \\
\kappa_{\HL_3} & = \tfrac{1}{15}a_3p_{1|1} - \tfrac{34}{315}a_2^2 - \tfrac{1}{63}a_2p_{1|0} + \tfrac{2}{105}p_{1|0}^2 - \tfrac{2}{105}a_2p_{1|1}^2 + \tfrac{2}{315}p_{1|0}p_{1|1}^2, \\
\kappa_{p_1^4}
& = 81 a_3^2 - 81 a_2^3  + 108a_2^2p_{1|0} - 54a_2 p_{1|0}^2 + 12p_{1|0}^3 + 216a_2 a_3 p_{1|1} - 108a_3 p_{1|0} p_{1|1}, \\
& \quad + 54 a_2^2 p_{1|1}^2 - 36a_2p_{1|0} p_{1|1}^2 + 6p_{1|0}^2 p_{1|1}^2 - 12 a_3 p_{1|1}^3 - a_2p_{1|1}^4.
\end{align*}
Setting $\lambda = p_{1|1}^2$, the first two equations show that $\lambda$, $\kappa_{p_1^2}$, $\kappa_{\HL_2}$ span the same subspace as $a_2$, $p_{1|0}$, $p_{1|1}^2$.
The last two can then be used in turn to express $a_3p_{1|1}$ and $a_3^2$ in terms of $\kappa_{p_1^2},\kappa_{p_1^4},\kappa_{\HL_2},\kappa_{\HL_3},\lambda$.
All algebraic relations among the generators \eqref{eq:gens} are consequences of the single relation
$$(a_3p_{1|1})^2 = (a_3^2)(p_{1|1}^2).$$
It follows that
$$H^*(B\aut(\tau_{\CP^2}^\RR)^e) = \QQ[\kappa_{p_1^2},\kappa_{p_1^4},\kappa_{\HL_2},\kappa_{\HL_3},\lambda]/J,$$
where $J$ is the principal ideal generated by the element $(a_3p_{1|1})^2 - (a_3^2)(p_{1|1}^2)$ rewritten in the new generators.
Let $B\aut(\tau_{\CP^2}^\RR)_L^e$ denote the classifying space of $\tau_{\CP^2}^\RR$-fibrations with trivializations of the Euler difference and the classes $\kappa_{\HL_2},\kappa_{\HL_3}$. The classes $\kappa_{\HL_2},\kappa_{\HL_3}$ form a regular sequence in the cohomology of $B\aut(\tau_{\CP^2}^\RR)^e$, so
$$H^*(B\aut(\tau_{\CP^2}^\RR)_L^e;\QQ) \cong \QQ[\kappa_{p_1^2},\kappa_{p_1^4},\lambda]/I,$$
where $I$ is the reduction of $J$ modulo $(\kappa_{\HL_2}, \kappa_{\HL_3})$. By rewriting $(a_3p_{1|1})^2 - (a_3^2)(p_{1|1}^2)$ in the new generators and multiplying with a suitable scalar, we find that $I$ is generated by the element
$$\lambda^4 -\tfrac{6304}{2023}\kappa_{p_1^2}\lambda^3 + \tfrac{35905}{14161}\kappa_{p_1^2}^2\lambda^2 +  \left(\tfrac{116}{289}\kappa_{p_1^2}^3 - \tfrac{1764}{289}\kappa_{p_1^4} \right)\lambda.$$
In particular, this shows that the kernel of the surjective map
$$H^*(B\aut(\tau_{\CP^2}^\RR)_L^e;\QQ) \to H^*(B\Isom^+(\CP^2);\QQ)$$
is the principal ideal generated by $\lambda$.

We have that $pd_{1|1} = -p_{1|1}$ and $pd_{1|0} = -2a_2-p_{1|0}$. The equations
\begin{align*}
21pd_{1|0} & = 4\kappa_{p_1^2} - 7\kappa_{ep_1} + 180\kappa_{\HL_2}, \\
45\kappa_{\HL_2} & = 6pd_{1|0} -pd_{1|1}^2,
\end{align*}
show that $pd_{1|0}$ and $\lambda$ are tautological and that $\lambda = 6pd_{1|0}$ if $\kappa_{\HL_2} = 0$.

For an arbitrary $\tau_{\CP^n}^\RR$-fibration over a space $B$ with trivial Euler difference and trivial $\kappa_{\HL_2}, \kappa_{\HL_3}$, the above shows that the kernel of
$$R^*(B) \to H^*(B\Isom^+(\CP^2);\QQ)$$
is the principal ideal generated by $pd_{1|0}$.

Imposing trivializations of $\kappa_{\HL_i}$ for $i>3$ by taking the homotopy fiber of a suitable map from $B\aut(\tau_{\CP^2}^\RR)_L^e$ to a product of Eilenberg-Mac Lane spaces will not change the cohomology in degree $4$, so the resulting space will have a $\tau_{\CP^2}^\RR$-fibration over it with trivial Euler difference and $\kappa_{\HL_i} = 0$ for all $i>1$, but with $pd_{1|0}\ne 0$. In particular, it has a non-vanishing Pontryagin difference.
\end{proof}

\end{document}